\documentclass{article}
\usepackage[utf8]{inputenc}
\usepackage[margin=1in]{geometry}
\usepackage{amsmath,amsthm,amssymb}
\usepackage{enumerate}
\usepackage{xfrac}
\usepackage[colorlinks=true,hidelinks]{hyperref}
    \AtBeginDocument{}
\usepackage{tikz} 
\usetikzlibrary{arrows}
\usetikzlibrary{patterns}
\usepackage{pgfplots} 
\usetikzlibrary{decorations.markings} 
\usetikzlibrary{3d,calc} 

\usepackage{bbold} 

\usepackage{fancyhdr}
\newcommand{\A}{\mathbb{A}}
\newcommand{\C}{\mathbb{C}}

\newcommand{\F}{\mathbb{F}}

\newcommand{\N}{\mathbb{N}}
\renewcommand{\P}{\mathbb{P}}
\newcommand{\Q}{\mathbb{Q}}

\newcommand{\U}{\mathcal{U}}

\newcommand{\X}{\mathcal{X}}

\newcommand{\Z}{\mathbb{Z}}

\newcommand{\Fr}{\text{Fr}} 

\renewcommand{\char}{\ensuremath{\operatorname{char}}}
\newcommand{\Spec}{\ensuremath{\operatorname{Spec}}}

\newcommand{\proj}{\ensuremath{\operatorname{proj}}}

\newcommand{\orb}{\mathcal{O}}

\newcommand*\cat[1]{{\tt #1}}

\newcommand{\Hom}{\ensuremath{\operatorname{Hom}}}

\newcommand{\Fun}{\ensuremath{\operatorname{Fun}}}

\newcommand{\Tr}{\ensuremath{\operatorname{Tr}}}
\newcommand{\Gal}{\ensuremath{\operatorname{Gal}}}
\newcommand{\Sym}{\ensuremath{\operatorname{Sym}}}

\newcommand{\eff}{\ensuremath{\operatorname{eff}}}

\newcommand{\Frob}{\ensuremath{\operatorname{Frob}}}

\newcommand{\supp}{\ensuremath{\operatorname{supp}}}

\newcommand{\len}{\ensuremath{\operatorname{len}}}
\newcommand{\invlim}{\displaystyle\lim_{\longleftarrow}}

\newcommand{\legen}[2]{\left (\frac{#1}{#2}\right )}

\newcommand{\roof}[3]{\tikz[baseline=10]{
  \node at (-1,0) (a) {$#1$};
  \node at (1,0) (b) {$#2$};
  \node at (0,1) (top) {$#3$};
  \draw[->] (top) -- (a);
  \draw[->] (top) -- (b);
}}
\newcommand{\roofx}[5]{\tikz[baseline=10]{
  \node at (-1,0) (a) {$#1$};
  \node at (1,0) (b) {$#2$};
  \node at (0,1) (top) {$#3$};
  \draw[->] (top) -- (a) node[left,pos=.2] {$#4$};
  \draw[->] (top) -- (b) node[right,pos=.2] {$#5$};
}}

\tikzset{->-/.style={decoration={markings,mark=at position #1 with {\arrow{>}}},postaction={decorate}}}

\newcommand{\ds}{\displaystyle}

\newtheorem{thm}{Theorem}[section]
\newtheorem{prop}[thm]{Proposition}

\newtheorem{lem}[thm]{Lemma}
\newtheorem{defn}[thm]{Definition}
\newtheorem{cor}[thm]{Corollary}

\newtheorem{rem}[thm]{Remark}
  \let\oldrem\rem
  \renewcommand{\rem}{\oldrem\normalfont}
\newtheorem{question}[thm]{Question}
\newtheorem{ex}[thm]{Example}
  \let\oldex\ex
  \renewcommand{\ex}{\oldex\normalfont}

\makeatletter
\let\@fnsymbol\@arabic
\makeatother

\title{Categorifying Zeta Functions for Quadratic Covers}
\author{Jon Aycock and Andrew Kobin\thanks{The second author was partially supported by the American Mathematical Society and the Simons Foundation.}} 
\date{\today}

\begin{document}

\setcounter{MaxMatrixCols}{20}

\maketitle

\rhead{\thepage}
\cfoot{}

\begin{abstract}
In various contexts, the zeta function of an object splits into a product of $L$-functions. We categorify this product formula for quadratic covers of objects in the following contexts: quadratic extensions of number fields, ramified double covers of algebraic curves, ramified double covers of topological spaces and Galois double covers of graphs. Our unified approach utilizes objective linear algebra in the abstract incidence algebra of each object, interpreted appropriately. We also provide several applications: for a hyperelliptic curve $C$ over a finite field, we prove a collection of combinatorial formulas relating the number of ramified, split and inert points on $C$ to the overall point count of $C$; and for a graph $G$, we deduce analogous combinatorial formulas for the numbers of split and inert primes in a Galois double cover $\widetilde{G}\rightarrow G$. We then use the formulas for graphs to deduce asymptotic counts of cycles in supersingular isogeny graphs and certain associated dual graphs of special fibers of Shimura curves. Finally, we analyze quadratic reciprocity from the perspective of zeta functions. 
\end{abstract}

\section{Introduction}
\label{sec:intro}

This paper offers a unified category theoretic perspective on factorizations of zeta functions in several contexts: number fields, curves over finite fields, Riemann surfaces, finite CW complexes and finite graphs. 

Let $K/\Q$ be a number field. Its Dedekind zeta function is defined as the formal expression 
$$
\zeta_{K}(s) = \sum_{\frak{a}\in I_{K}^{+}} \frac{1}{N(\frak{a})^{s}} = \sum_{n = 1}^{\infty} \frac{\#\{\frak{a}\mid N(\frak{a}) = n\}}{n^{s}}
$$
where $s$ is a complex number, $N = N_{K/\Q}$ is the field norm and $I_{K}^{+}$ denotes the set of nonzero ideals in the ring of integers $\orb_{K}$ of $K$. For example, $\zeta_{\Q}(s)$ is the Riemann zeta function. The formal Dirichlet series $\zeta_{K}(s)$ converges for complex numbers $s$ with real part greater than $1$ and, like the Riemann zeta function, $\zeta_{K}(s)$ admits a meromorphic continuation to the complex plane, satisfies a functional equation and has a well-formulated Riemann Hypothesis predicting the location of its nontrivial zeroes. Moreover, special values of $\zeta_{K}(s)$ contain rich arithmetic information which remains the subject of ongoing study. 

It is well-known that when $K/\Q$ is a quadratic extension, the zeta function factors in the ring of Dirichlet series as 
\begin{equation}\label{eq:quadzeta0}
\zeta_{K}(s) = \zeta_{\Q}(s)L(\chi,s)
\end{equation}
where $L(\chi,s)$ is the Dirichlet $L$-function 
$$
L(\chi,s) = \sum_{n = 1}^{\infty} \frac{\chi(n)}{n^{s}}
$$
associated to the Dirichlet character $\chi = \legen{D}{\cdot}$, with $D$ the discriminant of $K$. More generally, for any quadratic extension $K/F$, where $F$ is an arbitrary number field, the zeta function of $K$ splits as 
\begin{equation}\label{eq:quadzeta}
\zeta_{K}(s) = \zeta_{F}(s)L(\chi,s)
\end{equation}
where $\chi$ is the unique quadratic Hecke character on the group of fractional ideals $I_{F}$ of the ring of integers $\orb_{F}$ associated to the extension $K/F$ by class field theory. More generally, every zeta function $\zeta_{K}(s)$ splits as a product of $L$-functions coming from the representation theory of the Galois group of $K$. 

An analogue of formula (\ref{eq:quadzeta}) holds for curves over finite fields. To describe this, let $\F_{q}$ be a finite field and let $C$ be an algebraic curve over $\F_{q}$. Its zeta function is defined as the formal power series 
$$
Z(C,t) = \exp\left [\sum_{n = 1}^{\infty} \frac{\#C(\F_{q^{n}})}{n}t^{n}\right ]
$$
where $\#C(\F_{q^{n}})$ denotes the number of $\F_{q^{n}}$-points of $C$ and $\exp(t) = 1 + \sum_{n = 1}^{\infty} \frac{1}{n!}t^{n}$ is the formal exponential power series. For example, the zeta function of $\A^{1}$ is 
$$
Z(\A^{1},t) = \frac{1}{1 - qt}
$$
while the zeta function of $\P^{1}$ is 
$$
Z(\P^{1},t) = \frac{1}{(1 - t)(1 - qt)}. 
$$
By the Weil Conjectures, $Z(C,t)$ is a rational function, satisfies a functional equation and its zeroes and poles satisfy an analogue of the Riemann Hypothesis. 

Expanding further on the rationality property, it is well-known, though not at all obvious, that when $C$ is smooth, projective and geometrically integral of genus $g$, its zeta function can be written
$$
Z(C,t) = \frac{L(C,t)}{(1 - t)(1 - qt)}
$$
where $L(C,t)$ is the $L$-function of $C$, a degree $2g$ polynomial with integer coefficients. In other words, $Z(C,t)$ factors in the ring of formal power series as 
\begin{equation}\label{eq:Lfactor}
Z(C,t) = Z(\P^{1},t)L(C,t). 
\end{equation}

The point counting zeta function is only one generating function that encodes information about a curve $C/\F_{q}$. If $C$ lifts to a curve $\widetilde{C}$ in characteristic $0$, then the complex points of $\widetilde{C}$ have the structure of a Riemann surface, whose invariants can be extracted using generating functions. For instance, the Euler characteristic is encoded by the Macdonald polynomial $M(\widetilde{C},t) = (1 - t)^{-\chi(\widetilde{C})}$. In general, many formulas for $\chi$ can be encoded with Macdonald polynomials; for example, when $\pi : Y\rightarrow X$ is a branched cover of finite CW complexes of degree $n$, 
$$
\chi(Y) = n\chi(X) - r
$$
where $r = r(\pi)$ is the ramification number of $\pi$. Using Macdonald polynomials, 
\begin{equation}\label{eq:Macdonaldsplit}
M(Y,t) = M(X,t)^{n}(1 - t)^{r}. 
\end{equation}
For our purposes, $M(X,t)^{n - 1}(1 - t)^{r}$ plays the role of a ``topological $L$-function'' for the cover $\pi$. 

In a parallel direction, every finite graph $X$ admits an analogue of the zeta function of a curve over a finite field, called its Ihara zeta function $Z(X,u)$, defined by the formula 
\begin{equation}
    Z(X,u) = \prod_{\gamma} \frac{1}{1-u^{\len(\gamma)}},
\end{equation}
where the product is over so-called {\it prime cycles} in $X$ (see Section~\ref{sec:graphs}). The Ihara zeta function of a finite graph is a rational function in $u$, cf.~\cite[Thm.~2.5]{ter}, and for a $2$-to-$1$ Galois cover of graphs $Y \to X$, we can factor the zeta function of $Y$ as
\begin{equation}\label{eq:Iharasplit}
    Z(Y,u) = Z(X,u)L(X,\chi,u).
\end{equation}
Here, $\chi$ is the nontrivial character of the Galois group of $Y$ over $X$, and $L(X,\chi,u)$ is the corresponding {\it Artin--Ihara $L$-function}, which is rational as well \cite[Ch.~18]{ter}. More generally, as with the more number-theoretic zeta functions above, the zeta function of an arbitrary Galois cover factors as a product of Artin--Ihara $L$-functions \cite[Cor.~18.11]{ter}. 

\subsection{Categorification}

In this paper, we generalize the above formulas using objective linear algebra (see Section~\ref{sec:OLA} and Appendix~\ref{app:OLA} for an overview of this theory). Namely, for a double cover $\pi : Y\rightarrow X$ of an object $X$ (equal to $\Spec\orb_{F}$ for a number field $F$, an algebraic curve $C$ over an \emph{arbitrary} field, a finite CW complex or a finite graph), there is a formula for the zeta function (Dedekind zeta, point-counting zeta, Macdonald polynomial or Ihara zeta) of $Y$ as a product of $L$-functions: 
\begin{equation}\label{eq:genquad}
Z(Y) = Z(X)L(\pi). 
\end{equation}

The main goal of this paper is to use objective linear algebra to lift formula (\ref{eq:genquad}) to an equivalence of linear functors in the incidence algebra of an appropriate simplicial set $S(X)$ of ``effective $0$-cycles'' for $X$: 

\begin{thm}
\label{thm:introgen}
In the incidence algebra $I(S(X))$, there is an equivalence of linear functors 
$$
\pi_{*}\zeta_{Y} + \zeta_{X}*L(\pi)^{-} \cong \zeta_{X}*L(\pi)^{+}. 
$$
\end{thm}

Here, the choice of $S(X)$ depends on the object $X$: 
\begin{itemize}
    \item when $X = \Spec\orb_{F}$ for a number field $F$, $S(X)$ is the monoid of ideals of $\orb_{F}$ under multiplication;
    \item when $X$ is an algebraic curve, $S(X) = Z_{0}^{\eff}(X)$, the usual monoid of effective algebraic $0$-cycles on $X$ under addition;
    \item when $X$ is a finite CW complex, $S(X)$ is the simplicial complex $S_{\bullet}(X)$;
    \item when $X$ is a graph, $S(X)$ is the monoid of ``effective $0$-cycles" - similar to the case of curves, this is the free abelian monoid generated by the prime cycles (see Section~\ref{sec:graphs}).
\end{itemize}

In Theorem~\ref{thm:introgen}, the symbol $\pi_{*}$ is a pushforward map induced by the double cover $\pi : Y\rightarrow X$. The functors $L(\pi)^{+}$ and $L(\pi)^{-}$ in $\widetilde{I}(S(X))$ are defined in Section~\ref{sec:mainthm} in such a way that their images in the numerical incidence algebra satisfy $L(\pi)^{+} - L(\pi)^{-} = L(\pi)$, where $L(\pi)$ represents the coefficients of each of the ``$L$-functions'' in formulas (\ref{eq:quadzeta}), (\ref{eq:Lfactor}), (\ref{eq:Macdonaldsplit}) and (\ref{eq:Iharasplit}). Decategorifying in the right way recovers each of these formulas, as well as many others; see Remarks~\ref{rem:posetversion} and~\ref{rem:arbitrarydoublecovers}. 

Theorem~\ref{thm:introgen} also says that $L(\pi)$ is a {\it relative zeta function} for the cover $\pi$, in the terminology of \cite{ak}. Or to be more precise, it is a ``decomposed'' relative zeta function and becomes a relative zeta function after passing to the numerical incidence algebra $I_{\#}(S(X))$; see Remark~\ref{rem:card}. We plan to construct $L(\pi)$ directly as a single relative zeta function in a future article, using simplicial $G$-representations. See Appendix~\ref{app:Lpoly} for an overview of the construction in the case that $X$ is an elliptic curve. 

Finally, in Subsection~\ref{sec:motivic}, we give the following motivic analogue of Theorem~\ref{thm:introgen}. 

\begin{thm}[{Theorem~\ref{thm:motivicformula}}]
\label{thm:intromotivic}
For a double cover of curves $\pi : Y\rightarrow X$ over a field $k$, in the incidence algebra $I_{mot}(X)$ of the decomposition scheme $S^{\bullet}X$ associated to the monoid scheme $\coprod_{n\geq 0} \Sym^{n}X$, there is an equivalence of linear functors 
$$
\pi_{*}\zeta_{Y} + \zeta_{X}*L(\pi)^{-} \cong \zeta_{X}*L(\pi)^{+}. 
$$
\end{thm}

\subsection{Applications}

When $E$ is an elliptic curve over a finite field $\F_{q}$, the $L$-function of $E$ is of the form $L(E,t) = 1 - a_{q}t + qt^{2}$ where $a_{q} = q + 1 - \#E(\F_{q})$. Using Theorem~\ref{thm:introgen}, we deduce the formulas 
\begin{equation}\label{eq:aq}
a_{q} = i_{q}(E) - s_{q}(E) \quad\text{and}\quad \#\Sym^{n}C(\F_{q}) = \sum_{i + j = n} \left[(q^{i} + \ldots + q + 1)\sum_{\alpha\in\Sym^{j}\P^{1}(\F_{q})} \chi(\alpha)\right]
\end{equation}
where $s_{q}(E)$ and $i_{q}(E)$ are the number of ``split'' and ``inert'' points of the cover $E\rightarrow\P^{1}$ over $\F_{q}$ and $\chi$ is a certain function on the points of $\Sym^{j}\P^{1}$, which can be identified with the effective $0$-cycles of degree $j$ on $\P^{1}$. It is possible to deduce more complicated formulas for higher genus hyperelliptic curves in a similar fashion. 

Analogously, when $Y\xrightarrow{\pi} X$ is a Galois double cover of finite graphs, the ratio $\zeta_{Y}(u)/\zeta_{X}(u)$ of the graphs' Ihara zeta functions is an {\it Artin--Ihara $L$-function} $L(\pi,u) = \sum_{n = 0}^{\infty} f(n)u^{n}$ whose coefficients are determined by 
\begin{equation}\label{eq:loopcount}
f(n) = \sum_{\len(\alpha) = n} \chi(\alpha)
\end{equation}
where the sum is over all effective $0$-cycles $\alpha$ on $X$ and $\chi$ is a quadratic character on the primes of $X$ with values determined by the number of split and inert primes in the cover $\pi$. As with hyperelliptic curves, formula (\ref{eq:loopcount}) allows one to quickly compute loop counts in $Y$ from information stored on $X$. 

As a special case of this counting technique on graphs, we consider the graph $G_{p}(\ell)$ of supersingular $\ell$-isogenies, whose vertices are supersingular elliptic curves over $\overline{\F}_{p}$ and whose edges correspond to $\ell$-isogenies, for $\ell$ a prime distinct from $p$. Work of Ribet \cite{rib} shows that $G_{p}(\ell)$ admits a double cover by the dual graph $\widetilde{G}$ of the special fiber of a certain Shimura curve. We use formula (\ref{eq:loopcount}) to recover known asymptotics (cf.~\cite{aclsst}) for $G_{p}(\ell)$ and extend them to $\widetilde{G}$: 

\begin{thm}[{Theorem~\ref{thm:aclsst}, Corollary~\ref{cor:shimuraasymptotic}}]
\label{thm:shimuraasymptoticintro}
For distinct primes $p$ and $\ell$, let $G = G_{p}(\ell)$ be the graph of supersingular $\ell$-isogenies over $\overline{\F}_{p}$ and let $\widetilde{G} = \widetilde{G}_{p}(\ell)$ be the dual graph to the special fiber at $\ell$ of the Shimura curve $X_{0}^{p\ell}(1)$. When $p\equiv 1\pmod{12}$, the number of isogeny cycles of length $m$ in $G$ grows asymptotically as $\ell^{m}/2m$ and, if $m\geq 2$ is even, the number of prime cycles of length $m$ in $\widetilde{G}$ grows asymptotically as $\ell^{m}/m$. 
\end{thm}

\subsection{Relation to Previous Work}

This paper subsumes the authors' previous papers \cite{ak} and \cite{ak2} on quadratic number fields and hyperelliptic curves, respectively, and extends the techniques to new situations, notably Galois covers of finite graphs. 

\subsection{Organization}

The paper is organized as follows. In Section~\ref{sec:doublecov}, we recall basic properties of number fields, algebraic curves over finite fields, finite CW complexes and finite graphs, as well as each of their zeta and $L$-functions. In Section~\ref{sec:incalgs}, we review the definitions of decomposition sets and their incidence algebras, following \cite{gkt1,kob}. The proofs of Theorems~\ref{thm:introgen} and~\ref{thm:intromotivic} are given in Section~\ref{sec:mainthm}. In Section~\ref{sec:app} we prove formula (\ref{eq:aq}) relating the numbers of split and inert points of $E$ to $a_{q}(E)$ for an elliptic curve $E/\F_{q}$, as well as its graph theoretic analogue (\ref{eq:loopcount}) for a double cover of graphs $Y\rightarrow X$. Section~\ref{sec:ss} gives a brief overview of supersingular isogeny graphs and a proof of Theorem~\ref{thm:shimuraasymptoticintro}. Finally, in Section~\ref{sec:recip}, we describe quadratic reciprocity in the language of zeta functions of quadratic covers. 

The authors would like to thank their team members at the first Rethinking Number Theory (RNT) workshop - Karen Acquista, Changho Han and Alicia Lamarche - for their early contributions to this project and for many conversations along the way. The authors are grateful to the organizers of RNT - Heidi Goodson, Christelle Vincent and Mckenzie West - for the opportunity to begin this project, as well as to all the RNT participants. Finally, the second author would like to thank David Zureick-Brown for numerous suggestions that improved the article, Aly Deines for many productive discussions, and Bogdan Krstic for his help with the details of objective characteristic polynomials.

\section{Double Covers}
\label{sec:doublecov}

In this section, we survey four different types of finite covers: finite extensions of number fields, superelliptic curves (and more generally Galois covers of curves), topological covering spaces and Galois covers of graphs. In each setting, there is a natural notion of a zeta function which splits as a product of $L$-functions. We describe this splitting explicitly for double covers.

\subsection{Quadratic Number Fields}
\label{sec:quadnf}

Let $K/\Q$ be a number field with ring of integers $\orb_{K}$, norm map $N = N_{K/\Q}$ and Dedekind zeta function $\zeta_{K}(s)$. Then $\zeta_{K}(s)$ has the following product formula: 
$$
\zeta_{K}(s) = \prod_{\frak{p}} \frac{1}{1 - N(\frak{p})^{-s}}. 
$$
Here, the product is over all prime ideals $\frak{p}$. We denote the set of all ideals in $\orb_{K}$ by $I_{K}^{+}$ to distinguish it from $I_{K}$, which is commonly used for the group of fractional ideals. 

Now assume $K/F$ is an extension of number fields. For a fixed prime ideal $p\in I_{F}^{+}$ and a fixed prime ideal $\frak{p}\in I_{K}^{+}$ lying over $p$, let $e(\frak{p}\mid p)$ and $f(\frak{p}\mid p)$ denote the ramification index and inertia degree of $\frak{p}$, respectively. Recall that $p$ has one of the following splitting types: 
\begin{itemize}
    \item {\bf ramified} if $e(\frak{p}\mid p) > 1$ for some $\frak{p}$ lying over $p$;
    \item {\bf split} if $e(\frak{p}\mid p) = f(\frak{p}\mid p) = 1$ for all $\frak{p}$ lying over $p$;
    \item {\bf inert} if $p\orb_{K}$ is itself a prime ideal. 
\end{itemize}
When $F = \Q$ and the splitting types of all primes in the extension $K/\Q$ are known, one can factor $\zeta_{K}(s)$ as a Dirichlet series, that is, into local factors at each prime. This can also be done relatively, i.e.~in an extension $K/F$ with $F$ arbitrary. We state this for quadratic extensions below. 

\begin{prop}
\label{prop:quadsplitting}
Let $K/F$ be a quadratic extension of number fields. Then $\zeta_{K}(s)$ has the following prime factorization ``over $F$'': 
$$
\zeta_{K}(s) = \prod_{p\in I_{F}^{+}} \zeta_{K,p}(s)
$$
where
$$
\zeta_{K,p}(s) = \begin{cases}
    \ds\frac{1}{1 - N_{F/\Q}(p)^{-s}}, &\text{if $p$ is ramified}\\[1em]
    \ds\frac{1}{(1 - N_{F/\Q}(p)^{-s})^{2}}, &\text{if $p$ is split}\\[1em]
    \ds\frac{1}{1 - N_{F/\Q}(p)^{-2s}}, &\text{if $p$ is inert}.
\end{cases}
$$
\end{prop}

We can view each of these local factors as a generating function for ideals lying over $p$: 
\begin{align*}
    \zeta_{K,p}(s) &= \sum_{n = 0}^{\infty} \#\{\frak{a}\in I_{F}^{+} \mid N_{K/\Q}(\frak{a}) = N_{F/\Q}(p)^{n}\}p^{-ns}\\
        &= \begin{cases}
            1 + N_{F/\Q}(p)^{-s} + N_{F/\Q}(p)^{-2s} + \ldots, &\text{if $p$ is ramified}\\
            1 + 2N_{F/\Q}(p)^{-s} + 3N_{F/\Q}(p)^{-2s} + \ldots, &\text{if $p$ is split}\\
            1 + N_{F/\Q}(p)^{-2s} + N_{F/\Q}(p)^{-4s} + \ldots, &\text{if $p$ is inert}. 
        \end{cases}
\end{align*}

\begin{ex}
A well-known example is when $K = \Q(i)$. In this case, 
$$
p \text{ is } \begin{cases}
    &\text{ramified when } p = 2\\
    &\text{split when } p\equiv 1\!\pmod{4}\\
    &\text{inert when } p\equiv 3\!\pmod{4}. 
\end{cases}
$$
This means the zeta function of $\Q(i)$ can be written 
$$
\zeta_{\Q(i)}(s) = \frac{1}{1 - 2^{-s}}\times\prod_{p\equiv 1\,(4)} \frac{1}{(1 - p^{-s})^{2}}\times\prod_{p\equiv 3\,(4)} \frac{1}{1 - p^{-2s}}. 
$$
\end{ex}

\subsection{Hyperelliptic Curves}
\label{sec:hyperell}

The above description of zeta functions of quadratic extensions applies equally well in the setting of function fields. In this section, we describe the corresponding geometric picture, i.e.~using the point-counting zeta function of curves over a finite field.

For a smooth, projective, geometrically integral variety $X$ over $k = \F_{q}$, let $|X|$ denote the set of closed points of $X$. Then the zeta function of $X$ has the following product formula: 
$$
Z(X,t) = \prod_{x\in |X|} \frac{1}{1 - t^{\deg(x)}}
$$
where $\deg(x) = [k(x) : k]$ is the degree of $x$. We denote by $Z_{0}(X)$ the group of $0$-cycles on $X$, i.e.~the free abelian group generated by $|X|$. A $0$-cycle is {\it effective} if it is of the form $\alpha = \sum_{x\in |X|} a_{x}x$ for $a_{x}\geq 0$. The effective $0$-cycles on $X$ form a submonoid of $Z_{0}(X)$, denoted $Z_{0}^{\eff}(X)$. The zeta function of $X$ can also be written as a generating function for the effective $0$-cycles on $X$: 
$$
Z(X,t) = \sum_{\alpha\in Z_{0}^{\eff}(X)} t^{\deg(\alpha)}
$$
where $\deg : Z_{0}(X)\rightarrow\Z$ is the degree map sending $\alpha = \sum_{x} a_{x}x$ to $\deg(\alpha) = \sum_{x} a_{x}{\deg(x)}$.

\begin{rem}
By the number field/function field analogy, one can treat $\Spec\orb_{K}$ as an arithmetic curve and define $Z_{0}^{\eff}(\Spec\orb_{K})$ the same way, producing a monoid which is isomorphic to $I_{K}$. In fact, the effective $0$-cycles approach works for any scheme, so it should be considered the more natural choice. 
\end{rem}

Let $C$ be a hyperelliptic curve of genus $g\geq 1$, i.e.~a smooth, projective algebraic curve given by an equation of the form 
$$
C : y^{2} + h(x)y = f(x)
$$
for $f,h\in k[x]$ with $f$ monic, $\deg(f) = 2g + 1$ and $\deg(h)\leq g$. (If $\char k\not = 2$, we can take $h = 0$.) Restricting the projection $\P^{2}\rightarrow\P^{1}$ to $C$ determines a smooth map $\pi : C\rightarrow\P^{1}$ which is a degree $2$ cover ramified at $2g + 2$ points. 

\begin{rem}
Although it is not standard everywhere, in this paper we include elliptic curves in the class of hyperelliptic curves. 
\end{rem}

Fix a point $x\in |\P^{1}|$ and a point $y\in\pi^{-1}(x)$ on $C$. Let $e(y\mid x)$ denote the ramification index of $\pi$ at $y$ and let $f(y\mid x)$ denote the inertia degree, defined as $[k(y) : k(x)]$ where $k(x)$ (resp.~$k(y)$) is the field of definition of $x$ (resp.~$y$). Then $x$ has one of the following splitting types: 
\begin{itemize}
    \item {\bf ramified} if $e(y\mid x) > 1$ for any $y\in\pi^{-1}(x)$;
    \item {\bf split} if $e(y\mid x) = f(y\mid x) = 1$ for all $y\in\pi^{-1}(x)$;
    \item {\bf inert} if $\pi^{-1}(x)$ consists of a single closed point $y$ and $e(y\mid x) = 1$. 
\end{itemize}
In this situation, the Riemann--Hurwitz formula reads 
$$
2 - 2g = 4 - \sum_{y\in |C|} (e(y\mid x) - 1). 
$$
Compare this to the formulas in Section~\ref{sec:gencov}. 

\begin{prop}
\label{prop:quadsplitting}
Let $\pi : C\rightarrow\P^{1}$ be a hyperelliptic curve over $\F_{q}$. Then $Z(C,t)$ factors as 
$$
Z(C,t) = \prod_{x\in |\P^{1}|} Z_{x}(C,t)
$$
where
$$
Z_{x}(C,t) = \begin{cases}
    \ds\frac{1}{1 - t^{\deg(x)}}, &\text{if $x$ is ramified}\\[1em]
    \ds\frac{1}{(1 - t^{\deg(x)})^{2}}, &\text{if $x$ is split}\\[1em]
    \ds\frac{1}{1 - t^{2\deg(x)}}, &\text{if $x$ is inert}.
\end{cases}
$$
\end{prop}

\begin{proof}
The factorization of $Z(C,t)$ as a product over the closed points of $C$ is well known: 
$$
Z(C,t) = \prod_{y\in |C|} \frac{1}{1 - t^{\deg(y)}}. 
$$
For each $x\in|\P^{1}|$, write 
$$
Z_{x}(C,t) = \prod_{y\in\pi^{-1}(x)} \frac{1}{1 - t^{\deg(y)}}. 
$$
Then the description of $Z_{x}(C,t)$ follows from the definitions of ramified, split and inert points. This will also follow from the objective description of $Z(C,t)$ as the pushforward $\pi_{*}\zeta_{C}$ (see Section~\ref{sec:app}). 
\end{proof}

\subsection{Topological Covers}
\label{sec:gencov}

By the function field/Riemann surface analogy, we can replace $C/\F_{q}$ with a complex algebraic curve $C/\C$, or equivalently with a Riemann surface $X = C^{\operatorname{an}}$. Then (ramified) covers of algebraic curves correspond to (branched) covers of Riemann surfaces. Everything that follows works in the more general setting of finite-sheeted branched covers of finite CW complexes, so for the rest of this section, assume $X$ and $Y$ are finite CW complexes. 

First, suppose $\pi : Y\rightarrow X$ is an unbranched cover of degree $n\geq 2$. Then the Euler characteristics of $X$ and $Y$ satisfy the following well-known relation: 
\begin{equation}\label{eq:topEuler}
\chi(Y) = n\chi(X). 
\end{equation}

If $\pi : Y\rightarrow X$ is a branched cover, the relation between $\chi(X)$ and $\chi(Y)$ is 
$$
\chi(Y) = n\chi(X) - r,
$$
where $r$ is the ramification number of the cover. Call the irreducible components of the branch locus $X_{1},\ldots,X_{r}\subset X_{br}$ and for each component $Z_{i}\subseteq\pi^{-1}(X_{i})$, let $e(Y_{i})$ be the ramification index of $X_{i}$. Setting $e(Z) = 1$ for each irreducible component $Z\subseteq Y,Z\not\subseteq\pi^{-1}(X_{1}\cup\cdots\cup X_{r})$, we have 
$$
\sum_{Z\subseteq\pi^{-1}(X')} e(Z) = n
$$
for \emph{any} irreducible component $X'\subseteq X$. Then the precise analogue of formula (\ref{eq:topEuler}) for branched covers is 
\begin{equation}\label{eq:branchedEuler}
\chi(Y) = n\chi(X) - \sum_{Z\subseteq Y} (e(Z) - 1). 
\end{equation}
That is, $r = \sum (e(Z) - 1)$.

\subsection{Galois Covers of Graphs}
\label{sec:graphs}

Let $X = (V,\vec{E},o,t,\iota)$ be a finite, connected (bidirected) graph, where $V$ is a set of vertices, $\vec{E}$ is a set of directed edges, $o,t \colon \vec{E} \to V$ are the ``origin'' and ``terminus'' functions and $\iota \colon \vec{E} \to \vec{E}$ is an involution\footnote{Classically, the involution is assumed to be fixed-point free, but we want to allow half-loops.} satisfying $o(\iota(e)) = t(e)$ and $t(\iota(e)) = o(e)$ for all $e \in \vec{E}$. We think of the pair $e$ and $\iota(e)$ as a single undirected edge connecting the vertices $o(e)$ and $t(e)$. Note that $e = \iota(e)$ is possible under this definition, in which case $o(e) = t(e)$ and $e$ is referred to as a half-loop. 

A path of length $\ell$ on $X$ is an ordered list $\gamma$ of edges $[e_1\cdots e_\ell]$ with $o(e_{i+1}) = t(e_i)$ for each $1\leq i\leq\ell$. We say the path is {\it closed} if $o(e_1) = t(e_\ell)$; {\it reduced} if $e_{i+1} \neq \iota(e_i)$ for all $i$ and $e_1 \neq \iota(e_\ell)$; and it is a power of another path if it is just a shorter list $[e_1 \dots e_{\ell_0}]$ repeated some integer number of times. We say two closed, reduced paths are equivalent if they are related under the equivalence relation generated by the relations $[e_1 \cdots e_\ell] \sim [e_2 \cdots e_\ell e_1]$. 

A \emph{prime} on $X$ is an equivalence class of closed, reduced paths which are not powers of another path. Note that the lengths of primes are well-defined, since each equivalence class only contains paths of a single length. In addition, for each prime $\gamma = [e_1\cdots e_\ell]$, there is an opposite prime of the same length, $\iota(\gamma) = [\iota(e_\ell)\cdots\iota(e_1)]$.

We can then define the {\it Ihara zeta function} of a graph $X$ in terms of the following product formula:
\begin{equation}
    Z(X,u) = \prod_{\gamma \text{ prime}} \dfrac{1}{1-u^{\len(\gamma)}}.
\end{equation}
In analogy with the point-counting zeta function for varieties over $\F_{q}$, the Ihara zeta function counts closed paths in $X$. Explicitly, 
$$
Z(X,u) = \exp\left [\sum_{\ell = 1}^{\infty} \frac{N_{X}(\ell)}{\ell}u^{\ell}\right ]
$$
where $N_{X}(\ell)$ denotes the number of closed paths in $X$ of length $\ell$ without backtracks or tails \cite[4.5]{ter}. 

The Ihara zeta function has an alternative description as a (reciprocal) characteristic polynomial 
$$
Z(X,u) = \frac{1}{\det(1 - uT_{X})},
$$
where $T_{X}$ is the {\it edge adjacency matrix} for $X$ \cite[Cor.~11.5]{ter}. In particular, $Z(X,u)$ is rational. 

Let $X = (V_X,\vec{E}_X,o,t,\iota)$ and $Y = (V_Y,\vec{E}_Y,o,y,\iota)$ be two graphs. A morphism $f \colon Y \to X$ is a pair of functions $f_E \colon \vec{E}_Y \to \vec{E}_X$ and $f_V \colon V_Y \to V_X$ satisfying the relations $o(f_E(e)) = f_V(o(e))$, $t(f_E(e)) = f_V(t(e))$, and $\iota(f_E(e)) = f_E(\iota(e))$ for all $e \in \vec{E}_Y$.

For each vertex $v \in V_Y$, let $E_Y^+(v) = \{e \in \vec{E}_Y \mid o(e) = v\}$. We say that a morphism $f \colon Y \to X$ is a {\it covering map} if $f_V$ is surjective and, for each $v \in V_Y$, $f_E$ induces an isomorphism $E_Y^+(v) \to E_X^+(f_V(v))$. If, in addition, the functions $f_E$ and $f_V$ are both $d$-to-$1$, we say that $f$ is a cover of degree $d$.

Given a prime $\gamma = [e_1\dots e_\ell]$ in $X$ and a degree $2$ covering map $f : Y\rightarrow X$, there are two possibilities for its splitting type: we say $\gamma$ is 
\begin{itemize}
    \item \textbf{split} if, for each representative $e_1\dots e_\ell$ of $\gamma$, there are two distinct closed, reduced paths of length $\ell$ $\tilde{e}_{1,1}\dots \tilde{e}_{\ell,1}$ and $\tilde{e}_{1,2}\dots \tilde{e}_{\ell,2}$ in $Y$ with $f_E(\tilde{e}_{i,j}) = e_i$ for each $i,j$;
    \item \textbf{inert} if, for each representative $e_1\dots e_\ell$ of $\gamma$ and each choice of an edge $\tilde{e}_{1,1}$ with $f_E(\tilde{e}_{1,1})$, there is a unique closed, reduced path of length $2\ell$, $\tilde{e}_{1,1}\dots\tilde{e}_{\ell,1}\tilde{e}_{1,2}\dots\tilde{e}_{\ell,2}$ in $Y$, with $f_E(\tilde{e}_{i,j}) = e_i$ for each $i,j$ and $\tilde{e}_{i,1} \neq \tilde{e}_{i,2}$ for each $i$. In particular, there will be no \emph{closed} paths in $Y$ that lift $e_1\dots e_\ell$, but two distinct closed paths in $Y$ which lift the closed path $e_1\dots e_\ell e_1 \dots e_\ell$: one starting with $\tilde{e}_{1,1}$ and one starting with $\tilde{e}_{1,2}$.
\end{itemize}
There is no ramification in this setting (although see Section~\ref{sec:future}). We can factor $Z(Y,u)$ into a product over the primes of $X$ as in Sections~\ref{sec:quadnf} and~\ref{sec:hyperell}. 

\begin{prop}
    Let $f \colon Y \to X$ be a double cover of graphs. Then $Z(Y,u)$ has the following prime factorization ``over $X$":
    \begin{equation}
        Z(Y,u) = \prod_{\gamma} Z_\gamma(Y,u)
    \end{equation}
    where the product is over all primes $\gamma$ in $X$ and 
    \begin{equation}
        Z_\gamma(Y,u) = \begin{cases}
            \dfrac{1}{\left(1-u^{\len(\gamma)}\right)^2}, & \text{if } \gamma \text{ is split} \\[1em]
            \dfrac{1}{1-u^{2\len(\gamma)}}, & \text{if } \gamma \text{ is inert.}
        \end{cases}
    \end{equation}
\end{prop}

When $f \colon Y \to X$ is a quadratic cover of graphs and $\gamma$ is a prime of $X$ that splits in $Y$, we label the two primes of $Y$ that lie over $\gamma$ as $\eta$ and $\overline{\eta}$. The formulas proven in Section~\ref{sec:mainthm} will not depend on this labeling, although their proofs will.

\section{Incidence Algebras and Objective Linear Algebra}
\label{sec:incalgs}

In this section, we define an incidence algebra of $0$-cycles in each of our four settings: number fields, varieties over $\F_{q}$, finite CW complexes and graphs. For most of these, this is accomplished by first attaching a monoid to each object, then invoking a general construction of incidence algebras for monoids, which is reviewed below. To integrate CW complexes, we attach a more general simplicial set, which then requires the more general theory of decomposition sets, objective linear algebra and objective incidence algebras. This theory is reviewed in Subsections~\ref{sec:decompset}, \ref{sec:OLA} and~\ref{sec:objinc}, following \cite{gkt1,kob}. 

\subsection{Incidence Algebras for Monoids}
\label{sec:monoids}

Fix a field $k$. We say a monoid $M$ is {\it locally finite} if for every $x\in M$, there are only finitely many ways of factoring $x = ab$ for $a,b\in M$. 

Following \cite{gkt1} and \cite{kob}, the incidence algebra of a locally finite monoid is defined by first defining its incidence coalgebra, then taking the dual. 

\begin{defn}
The {\bf incidence coalgebra} of a locally free monoid $M$ with unit $1$ is the free $k$-vector space 
$$
C(M) = \bigoplus_{x\in M} kx
$$
together with comultiplication and counit
\begin{align*}
    \Delta : C(M) &\longrightarrow C(M)\otimes_{k}C(M)\\
    x &\longmapsto \sum_{ab = x} a\otimes b\\
    \delta : C(M) &\longrightarrow k\\
      x &\longmapsto \delta_{x1}
\end{align*}
where $\delta_{xy}$ is the Kronecker delta symbol. Dually, the {\bf incidence algebra} of $M$ is the $k$-algebra $I(M) = \Hom_{k}(C(M),k)$, with multiplication defined by extensions
\begin{align*}
    * : I(M)\otimes_{k}I(M) &\longrightarrow I(M)\\
    f\otimes g &\longmapsto \left (f*g : x\mapsto \sum_{ab = x} f(a)g(b)\right )
\end{align*}
(extended linearly) and unit $\delta\in I(M)$. 
\end{defn}

For a locally finite monoid $M$, the incidence algebra $I(M)$ is always associative and unital; however, it need not be commutative. 

There is a distinguished element $\zeta\in I(M)$, called the {\it zeta function} of $M$, defined by $\zeta(x) = 1$ for all $x\in M$. The principle of {\it M\"{o}bius inversion} \cite{rot,gkt2} says that $\zeta$ is an invertible element of $I(M)$, whose inverse is called the {\it M\"{o}bius function} and is defined by the following recursion: 
$$
\zeta^{-1} = \mu : x \longmapsto \begin{cases}
    1, & x = 1\\
    -\sum_{ab = x} \mu(a), & x\not = 1.
\end{cases}
$$

Perhaps unsurprisingly, the zeta functions in our different settings correspond to zeta elements in the incidence algebras of certain monoids, which we define in the next few sections. 

Alternatively, these incidence algebras can be constructed by attaching a poset to our objects and taking the {\it reduced incidence algebra} of this poset. By the general theory of decomposition sets, decalage and CULF functors, these two approaches yield the same incidence algebra; cf.~\cite{gkt1,gkt5,kob}. 

\subsection{The Incidence Algebra of a Number Field}
\label{sec:incalgfield}

\renewcommand{\P}{\mathbb{P}}

Let $K/\Q$ be a number field with ring of integers $\orb_{K}$. The set of ideals in $\orb_{K}$ forms a locally finite monoid under multiplication, written $I_{K}^{\times}$. When $K = \Q$, this is just the monoid of natural numbers $\N^{\times}$. Let $I(K) = I(I_{K}^{\times})$ be the incidence algebra of this monoid. 

For $K = \Q$, the incidence algebra $I(\Q) = I(\N^{\times})$ is isomorphic to the algebra of formal Dirichlet series: 
\begin{align*}
    I(\N^{\times}) &\xrightarrow{\;\sim\;} DS(k)\\
    f &\longmapsto \sum_{n = 1}^{\infty} \frac{f(n)}{n^{s}}
\end{align*}
and the zeta element $\zeta\in\widetilde{I}(\N,\mid)$ corresponds to the Riemann zeta function: 
$$
\zeta \longleftrightarrow \sum_{n = 1}^{\infty} \frac{1}{n^{s}}. 
$$

For a number field $K/\Q$, the field norm induces a map of monoids $N : I_{K}^{\times}\rightarrow \N^{\times}$. The resulting map $I(K)\rightarrow I(\Q)\cong DS(k)$ can be used to construct a Dirichlet series for any $f\in I(K)$: 
$$
f \longmapsto F(s) = \sum_{\frak{a}\in I_{K}^{\times}} \frac{f(\frak{a})}{N(\frak{a})^{s}}. 
$$
In particular, the zeta element $\zeta\in I(K)$ maps to the Dedekind zeta function of $K$ under this norm map: 
$$
\zeta \longmapsto \sum_{\frak{a}\in I_{K}^{\times}} \frac{1}{N(\frak{a})^{s}} = \sum_{n = 1}^{\infty} \frac{\#N^{-1}(n)}{n^{s}}. 
$$
More generally, for an extension of number fields $K/F$, there is a notion of pushforward on incidence algebras which lifts this Dirichlet series machine to the relative setting (see Example~\ref{ex:pushfwd}). 

\subsection{The Incidence Algebra of Effective $0$-Cycles for a Variety}
\label{sec:0cyc}

Next, we construct the incidence algebra of effective $0$-cycles for a variety $X$ over a field $F$. Let $Z_{0}^{\eff}(X)$ be the set of effective $0$-cycles on $X$. This has the structure of a locally finite monoid under addition, so we can apply the construction from Section~\ref{sec:monoids}. Fix a field of coefficients $k$ (usually $k = \C$ or $\Q_{\ell}$ for $\ell\nmid q$). The {\it incidence algebra of effective $0$-cycles} of $X$, hereafter the incidence algebra of $X$, is 
$$
I(X) := I(Z_{0}^{\eff}(X)) = \Hom_{k}(C(Z_{0}^{\eff}(X)),k),
$$
with convolution product  
$$
(f*g)(\alpha) = \sum_{\beta + \gamma = \alpha} f(\beta)g(\gamma). 
$$
The {\it zeta function} of $X$ is the element $\zeta\in I(X)$ defined by $\zeta(\alpha) = 1$ for every $\alpha\in Z_{0}^{\eff}(X)$. 

\begin{ex}
When $F = \F_{q}$ and $X = \Spec\F_{q}$, $Z_{0}^{\eff}(\Spec\F_{q})$ is isomorphic to the monoid $\N_{0}$ under addition and as a result, its incidence algebra is isomorphic to the ring of power series: 
$$
I(\Spec\F_{q}) \xrightarrow{\;\sim\;} k[[t]], \quad f \longmapsto \sum_{n = 0}^{\infty} f(n)t^{n}. 
$$
\end{ex}

More generally, for any variety $X/\F_{q}$, $I(X)$ is isomorphic to a completed tensor product of power series rings (see Appendix~\ref{app:tensor}): 
$$
I(X) \xrightarrow{\;\sim\;} \widehat{\bigotimes_{x\in |X|}} k[[t_{x}]], \quad f \longmapsto \bigotimes_{x\in |X|} \sum_{n = 0}^{\infty} f(nx)t_{x}^{n}. 
$$
Notice that under the map $\widehat{\bigotimes}_{x\in |X|} k[[t_{x}]]\rightarrow k[[t]]$ induced by $t_{x}\mapsto t^{\deg(x)}$, $\zeta\in I(X)$ is sent to the usual point-counting zeta function $Z(X,t)$: 
$$
\zeta = \bigotimes_{x\in |X|} \sum_{n = 0}^{\infty} t_{x}^{n} \longmapsto \prod_{x\in |X|} \sum_{n = 0}^{\infty} t^{n\deg(x)} = \prod_{x\in |X|} \frac{1}{1 - t^{\deg(x)}} = Z(X,t). 
$$
This mapping $\widehat{\bigotimes}_{x\in |X|} k[[t_{x}]] \rightarrow k[[t]]$, which constructs a generating function for every $f\in I(X)$, is another special case of a pushforward between incidence algebras (see Example~\ref{ex:pushfwd}). We will see that formula (\ref{eq:Lfactor}) is a statement about the pushforward of $\zeta\in\widetilde{I}(C)$ along a map $C\rightarrow\P^{1}$.

\subsection{The Incidence Algebra of Effective $0$-cycles for a Graph}
\label{sec:incalggraph}

Similarly, for a graph $X$, let $Z_0^{\text{eff}}(X) = \N_0 \{\text{primes of } X\}$ be the free commutative monoid generated by the primes of $X$. This is a locally finite monoid, so we can apply the construction from Section~\ref{sec:monoids}. The \emph{incidence algebra of effective $0$-cycles} of $X$, hereafter the incidence algebra of $X$, is
$$
I(X) := I(Z_{0}^{\eff}(X)) = \Hom_{k}(C(Z_{0}^{\eff}(X)),k),
$$
with convolution product  
$$
(f*g)(\alpha) = \sum_{\beta + \gamma = \alpha} f(\beta)g(\gamma). 
$$
The {\it zeta function} of $X$ is the element $\zeta\in I(X)$ defined by $\zeta(\alpha) = 1$ for every $\alpha\in Z_{0}^{\eff}(X)$. 

\begin{ex}
When $X = C_\ell$ is a cycle graph, meaning a connected graph with $\ell$ vertices, each connected to the next in a cycle by $\ell$ undirected edges, then $X$ has only two primes, which we will denote $\gamma$ and $\gamma^\iota$. Thus $Z_{0}^{\eff}(C_\ell)$ is isomorphic to the monoid $\N_{0} \times \N_0$ under addition and as a result, its incidence algebra is isomorphic to the ring of power series: 
$$
I(C_\ell) \xrightarrow{\;\sim\;} k[[t_1,t_2]], \quad f \longmapsto \sum_{a,b \geq 0} f(a\gamma + b\gamma^\iota)t_1^{a}t_2^{b}. 
$$
\end{ex}

More generally, $I(X)$ is a completed tensor product of power series rings, as in Section~\ref{sec:0cyc} (again, see Appendix~\ref{app:tensor}). Likewise, formula (\ref{eq:Iharasplit}) is a consequence of the pushforward construction of Example~\ref{ex:pushfwd} applied to a double cover $Y\rightarrow X$.

\subsection{Decomposition Sets}
\label{sec:decompset}

It is possible to construct incidence algebras for a far more general class of objects than monoids. For example, Leroux and others \cite{ler,cll,dur} construct incidence algebras for so-called M\"{o}bius categories, which include locally finite monoids (and posets) as a special case via their nerves. Other examples of naturally-occurring incidence algebras include: the Hopf algebra of rooted trees \cite{but,ck}, the Fa\`{a} di Bruno bialgebra \cite{joy,dur} and the Hall algebra of an exact category \cite{dk}. 

All of the above situations have been subsumed by the concept of a {\it decomposition space}, introduced and studied by G\'{a}lvez-Carrillo, Kock and Tonks in \cite{gkt1,gkt2,gkt3,gkt-hla,gkt5}. In this paper, we will only need the notion of a {\it decomposition set}, so in this section we will explain the relevant definitions and properties from \cite{gkt1} in the category of simplicial sets. See Appendix~\ref{app:OLA} for a brief summary of the general theory. 

Fix a field of coefficients $k$ and let $S$ be a simplicial set, with face maps $d_{i} : S_{n}\rightarrow S_{n - 1}$ and degeneracy maps $s_{i} : S_{n}\rightarrow S_{n + 1}$. 

\begin{defn}
Let $S : \Delta^{op}\rightarrow\cat{Set}$ be a simplicial set which is locally of finite length (cf.~\cite{gkt2} for a precise definition). The {\bf numerical incidence coalgebra} of $S$ is the free $k$-vector space $C(S)$ on the set $S_{1}$ of $1$-simplices of $S$, together with the comultiplication map 
\begin{align*}
    \Delta : C(S) &\longrightarrow C(S)\otimes C(S)\\
    x &\longmapsto \sum_{\substack{\sigma\in S_{2} \\ d_{1}\sigma = x}} d_{2}\sigma\otimes d_{0}\sigma
\end{align*}
and counit 
\begin{align*}
    \delta : C(S) &\longrightarrow k\\
    x &\longmapsto \begin{cases}
      1, &\text{if $x$ is degenerate}\\
      0, &\text{if $x$ is nondegenerate}.
    \end{cases}
\end{align*}
The {\bf numerical incidence algebra} of $S$ is the dual vector space $I(S) = \Hom_{k}(C(S),k)$ together with the multiplication map 
\begin{align*}
    I(S)\otimes I(S) &\longrightarrow I(S)\\
    f\otimes g &\longmapsto \left (f*g : x\mapsto \sum_{\substack{\sigma\in S_{2} \\ d_{1}\sigma = x}} f(d_{2}\sigma)g(d_{0}\sigma)\right )
\end{align*}
and unit $\delta$. 
\end{defn}

To distinguish it from the abstract incidence algebra to be defined below, we will sometimes denote the numerical incidence algebra of $S$ by $I_{\#}(S)$. 

For any simplicial set $S$, the {\it zeta function} of $S$ is the element $\zeta\in I(S)$ defined by $\zeta : x\mapsto 1$ for all $x\in S_{1}$ and extended linearly. In general, $I(S)$ is neither associative nor unital as a $k$-algebra. However, when $S$ is a {\it decomposition set}, $I(S)$ is an associative, unital $k$-algebra \cite[Sec.~5.3]{gkt1}. Following {\it loc.~cit.}, we say a morphism $g : [m]\rightarrow [n]$ in the category $\Delta$ is {\it active} if $g(0) = 0$ and $g(m) = n$, and we say $g$ is {\bf inert} if $g(i + 1) = g(i) + 1$ for all $0\leq i\leq m - 1$. 

\begin{defn}
A {\bf decomposition set} is a simplicial set $S : \Delta^{op}\rightarrow\cat{Set}$ that takes any pushout diagram in $\Delta$ of the form 
\begin{center}
\begin{tikzpicture}[scale=2]
  \node at (0,1) (a) {$[p]$};
  \node at (1,1) (b) {$[m]$};
  \node at (0,0) (c) {$[\ell]$};
  \node at (1,0) (d) {$[n]$};
  \draw[->] (b) -- (a);
  \draw[->] (c) -- (a);
  \draw[->] (d) -- (b) node[right,pos=.5] {$f$};
  \draw[->] (d) -- (c) node[above,pos=.5] {$g$};
  \draw (.2,.7) -- (.3,.7) -- (.3,.8);
\end{tikzpicture}
\end{center}
where $f$ is inert and $g$ is active, to a pullback diagram in $\cat{Set}$: 
\begin{center}
\begin{tikzpicture}[scale=2]
  \node at (0,1) (a) {$S_{p}$};
  \node at (1,1) (b) {$S_{m}$};
  \node at (0,0) (c) {$S_{\ell}$};
  \node at (1,0) (d) {$S_{n}$};
  \draw[->] (a) -- (b);
  \draw[->] (a) -- (c);
  \draw[->] (b) -- (d) node[right,pos=.5] {$f^{*}$};
  \draw[->] (c) -- (d) node[above,pos=.5] {$g^{*}$};
  \draw (.2,.7) -- (.3,.7) -- (.3,.8);
\end{tikzpicture}
\end{center}
\end{defn}

Note that locally finite monoids are examples of {\it M\"{o}bius categories}, which are special types of decomposition sets. Therefore all incidence algebras constructed so far are associative and unital.

\subsection{Objective Linear Algebra}
\label{sec:OLA}

Informally, objective linear algebra is ``linear algebra with sets''. In lieu of a full description of the theory, which already appears in \cite{gkt-hla,kob,ak}, here is a dictionary of some linear algebra terms and their objective counterparts. 

\begin{center}
\begin{tabular}{|c|c|}
    \hline
    {\bf Linear} & {\bf Objective}\\
    \hline\hline
    field of scalars $k$ & the category $\cat{Set}$\\
    scalar addition $+$ & coproduct $\coprod$\\
    scalar multiplication & product $\times$\\
    a basis $B$ & a set $B$\\
    a vector $v$ in the basis $B$ & a set map $v : X\rightarrow B$\\
    the vector space with basis $B$ & the slice category $\cat{Set}_{/B}$\\
    vector addition $v + w$ & coproduct $v\amalg w : X\coprod Y\rightarrow B$\\
    scalar multiplication $av$ & $A\times (v : X\rightarrow B) := (A\times X\xrightarrow{\operatorname{id}\times v} A\times B\xrightarrow{\proj_{B}}B)$\\
    a matrix $M$ & a span $\roofx{B}{C}{M}{v}{w}$\\
    the linear map with matrix $M$ & the linear functor $w_{!}v^{*} : \cat{Set}_{/B}\rightarrow\cat{Set}_{/C}$\\
    matrix multiplication & span composition\\
    tensor product $V\otimes W$ & symmetric tensor $\cat{Set}_{/B}\otimes\cat{Set}_{/C} := \cat{Set}_{/B\times C}$\\
    dual space $V^{*} = \Hom(V,k)$ & functor space $(\cat{Set}_{/B})^{*} := \Fun(\cat{Set}_{/B},\cat{Set})$\\
    \hline
\end{tabular}
\end{center}

Here, for any span (as pictured), the composition $w_{!}v^{*}$ in the definition of linear functor consists of the pullback functor $v^{*} : \cat{Set}_{/B}\rightarrow\cat{Set}_{/M},(X\rightarrow B)\mapsto (X\times_{B}M\rightarrow M)$ and the pushforward functor $w_{!} : \cat{Set}_{/M}\rightarrow\cat{Set}_{/C},(Y\rightarrow M)\mapsto (Y\rightarrow M\xrightarrow{w}C)$. 

To recover the objects in the left column of the table, one can often apply the cardinality functor to objects in the right column. For example, if $v : X\rightarrow B$ is an objective vector with finite fibres, then the collection of its fibre cardinalities defines a vector $\sum_{b\in B} |v^{-1}(b)|b$ in the $k$-vector space spanned by $B$. 

\subsection{The Objective Incidence Algebra}
\label{sec:objinc}

Abstract incidence algebras are defined at the objective level for any {\it locally finite decomposition set} in \cite{gkt1}. For such a decomposition set $S$, the incidence algebra $I(S)$ is an associative, unital monoid object in the objective linear algebra category $\cat{LIN}$. In particular, in each of our four settings the associated simplicial set is a locally finite decomposition set and therefore admits an objective incidence algebra. 

\begin{defn}
For a simplicial set $S$, the {\bf incidence coalgebra} of $S$ is the objective vector space $C(S) = \cat{Set}_{/S_{1}}$ equipped with a  comultiplication linear functor $\Delta : C(S)\rightarrow C(S)\otimes C(S)$, represented by the span 
$$
\Delta = \left (\roofx{S_{1}}{S_{1}\times S_{1}}{S_{2}}{d_{1}}{(d_{2},d_{0})}\right ), 
$$
as well as a counit linear functor $\delta : C(S)\rightarrow\cat{Set}$, represented by the span 
$$
\delta = \left (\roofx{S_{1}}{*}{S_{0}}{s_{0}}{}\right ). 
$$
The {\bf incidence algebra} of $S$ is the dual objective vector space $I(S) := C(S)^{*} = \Fun(\cat{Set}_{/S_{1}},\cat{Set})$, equipped with unit $\delta$ and a multiplication linear functor $m : I(S)\otimes I(S)\rightarrow I(S)$ defined as follows. For $f,g\in I(S)$, their product $m(f,g)$ is the composition $m(f,g) : \cat{Set}_{/S_{1}}\xrightarrow{\Delta}\cat{Set}_{/S_{1}}\otimes\cat{Set}_{/S_{1}}\xrightarrow{f\otimes g}\cat{Set}\otimes\cat{Set}\xrightarrow{\sim}\cat{Set}$. This may also be represented by the outer span in the diagram 
$$
m(f,g) = \left ( \tikz[xscale=3,yscale=1.7,baseline=30]{
        \node at (0,0) (a) {$S_{1}$};
        \node at (1,0) (b) {$S_{1}\times S_{1}$};
        \node at (2,0) (c) {$*$};
        \node at (.5,.8) (x) {$S_{2}$};
        \node at (1.5,.8) (y) {$S_{1}\times S_{1}$};
        \node at (1,1.6) (z) {$P$};
        \draw[->] (z) -- (x);
        \draw[->] (z) -- (y);
        \draw[->] (x) -- (a) node[left,pos=.3] {$d_{1}$};
        \draw[->] (x) -- (b) node[right,pos=.1] {$(d_{2},d_{0})$};
        \draw[->] (y) -- (b) node[right,pos=.8] {$f\times g$};
        \draw[->] (y) -- (c);
    }\right )
$$
where the top square is a pullback square. 
\end{defn}

When $S$ is a decomposition set, \cite[Thm.~7.4]{gkt1} shows that $I(S)$ is an associative, unital monoid object in the category of objective vector spaces and linear functors; that is, $I(S)$ is an objective algebra over $\cat{Set}$. The span 
$$
\zeta = \left (\roofx{S_{1}}{*}{S_{1}}{id}{}\right )
$$
defines a linear functor $\zeta\in I(S)$, called the {\it zeta functor} of $S$. 

\begin{ex}
\label{ex:pushfwd}
One important construction we will need in later proofs is the notion of pushforward between incidence algebras. Suppose $S$ and $T$ are decomposition sets and $F : T_{1}\rightarrow S_{1}$ is a set map on their $1$-simplices. This induces a pushforward map between their incidence algebras $F_{*} : I(T)\rightarrow I(S)$ defined on spans as follows: 
$$
\varphi = \left (\roofx{T_{1}}{*}{M}{f}{}\right ) \quad \longmapsto \quad F_{*}\varphi = \left (\roofx{S_{1}}{*}{M}{F\circ f}{}\right ). 
$$
If $F : T_{1}\rightarrow S_{1}$ is of the form $F = G_{1}$ for a simplicial map $G : T\rightarrow S$, we will write $F_{*}$ or $G_{*}$ to mean this pushforward of spans. 
\end{ex}

\subsection{The Incidence Algebra for a CW Complex}
\label{sec:incalgCW}

For a CW complex $X$, let $S_{\bullet}(X)$ be the associated simplicial complex, which is a decomposition set by construction. We define the incidence algebra for $X$ to be incidence algebra of $S_{\bullet}(X)$. The numerical incidence algebra in this case is the $k$-algebra $I_{\#}(X) = I_{\#}(S_{\bullet}(X)) = \Hom_{k}(k[S_{1}(X)],k)$, with convolution product 
$$
(f*g)(\alpha) = \sum_{\substack{\sigma\in S_{2}(X) \\ d_{1}\sigma = \alpha}} f(d_{2}\sigma)g(d_{0}\sigma). 
$$
Here, $d_{0},d_{1},d_{2} : S_{2}(X)\rightarrow S_{1}(X)$ are the face maps which pick off the faces of a $2$-simplex $\sigma$ as follows: 
$$
\tikz[baseline=5]{
    \draw[thick] (-.7,0) -- (0,1) node[left,pos=.6] {$\beta$};
    \draw[thick] (0,1) -- (.7,0) node[right,pos=.4] {$\gamma$};
    \draw[thick] (.7,0) -- (-.7,0) node[below,pos=.5] {$\alpha$};
    \node at (0,.4) {\large $\sigma$};
    \node at (-.9,-.1) {0};
    \node at (0,1.3) {1};
    \node at (.9,-.1) {2};
} \qquad\qquad \begin{matrix}
    & d_{0} : \sigma \longmapsto \gamma\\
    & d_{1} : \sigma \longmapsto \alpha\\
    & d_{2} : \sigma \longmapsto \beta.\\
    &
\end{matrix}
$$
As usual, write $\zeta_{X}$ for the {\it zeta functor} and $|\zeta_{X}|$ for its decategorification in $I_{\#}(X)$. 

\begin{ex}
For $X = *$, the incidence algebra $I(*)$ is isomorphic to $\cat{Set}$, while the numerical incidence algebra is isomorphic to the ring of power series $k[[t]]$ and the zeta function $|\zeta_{*}|\in I_{\#}(*)$ corresponds to the geometric power series $(1 - t)^{-1}\in k[[t]]$. Notice that in this case, 
$$
|\zeta_{*}| = (1 - t)^{-\chi(*)}. 
$$
\end{ex}

More generally, the zeta functor of a CW-complex $X$ pushes forward along the terminal map $t : X\rightarrow *$ to an objective linear functor $t_{*}\zeta_{X}\in I(*)$ whose decategorification $|t_{*}\zeta_{X}|\in I_{\#}(*)$ is the Macdonald polynomial $(1 - t)^{-\chi(X)}$. For a covering map $\pi : Y\rightarrow X$, we will see that formula (\ref{eq:Macdonaldsplit}) is a consequence of a splitting of $\pi_{*}\zeta_{Y}$ in $I(X)$.

\section{Proof of the Main Theorem}
\label{sec:mainthm}

For a pair of objects $X$ and $Y$, let $S$ and $T$ denote their respective incidence algebras, as defined in Section~\ref{sec:incalgs}. Fix a quadratic cover $\pi : Y\rightarrow X$, interpreted appropriately in each of our four settings, and let $\pi_{*} : I(Y)\rightarrow I(X)$ be the pushforward on incidence algebras induced by $\pi$. 

\subsection{Local Theorem}
\label{sec:mainthmloc}

For this section, fix a ``point'' $x$ in $X$, interpreted as:
\begin{itemize}
    \item a prime ideal in $\orb_{F}$ when $X = \Spec\orb_{F}$;
    \item a closed point when $X$ is an algebraic curve over a field;
    \item a point when $X$ is a finite CW complex;
    \item a prime loop when $X$ is a graph. 
\end{itemize}
If $S = S(X)$ is the simplicial set attached to $X$, let $S_{x}$ be the simplicial subset ``supported at $x$'' , meaning: 
\begin{itemize}
    \item when $X = \Spec\orb_{F}$, $S$ consists of simplices of powers of $x$;
    \item when $X$ is an algebraic curve, $S$ consists of simplices of effective $0$-cycles supported at $x$;
    \item when $X$ is a finite CW complex, $S$ consists of simplices containing $x$;
    \item when $X$ is a graph, $S$ consists of equivalence classes of powers of $x$. 
\end{itemize}

Fix a quadratic cover $\pi : Y\rightarrow X$ and let $T_{x}$ be the simplicial subset of $T$ consisting of simplices with images in $S_{x}$. Also let $\zeta_{X,x}$ and $\zeta_{Y,x}$ be the zeta functions in $I(S_{x})$ and $I(T_{x})$, respectively. 

\begin{thm}[{Local Version of Theorem~\ref{thm:introgen}}]
\label{thm:mainthmloc}
For each point $x$ in $X$, in the incidence algebra $I(S_{x})$, there is an equivalence of linear functors 
\begin{equation}\label{eq:local}
\pi_{*}\zeta_{Y,x} + \zeta_{X,x}*L_{x}(\pi)^{-} \cong \zeta_{X,x}*L_{x}(\pi)^{+}. 
\end{equation}
\end{thm}

\begin{proof}
The terms in formula (\ref{eq:local}) are defined as follows. First, $\pi_{*}\zeta_{Y,x},L_{x}(\pi)^{+}$ and $L_{x}(\pi)^{-}$ are represented by the spans 
$$
\pi_{*}\zeta_{Y,x} = \left (\roofx{S_{x,1}}{*}{T_{x,1}}{\pi_{1}}{}\right ), \quad L_{x}(\pi)^{+} = \left (\roofx{S_{x,1}}{*}{S_{x,1}^{+}}{j_{x}^{+}}{}\right ), \quad L(\pi)^{-} = \left (\roofx{S_{x,1}}{*}{S_{x,1}^{-}}{j_{x}^{-}}{}\right )
$$
where $S_{x,1}^{+},S_{x,1}^{-},j_{x}^{+}$ and $j_{x}^{-}$ are each defined below. Then the convolutions in formula (\ref{eq:local}) are given by the span convolutions 
$$
\zeta_{X,x}*L_{x}(\pi)^{+} = \left ( \tikz[xscale=3,yscale=1.7,baseline=30]{
        \node at (0,0) (a) {$S_{x,1}$};
        \node at (1,0) (b) {$S_{x,1}\times S_{x,1}$};
        \node at (2,0) (c) {$*$};
        \node at (.5,.8) (x) {$S_{x,2}$};
        \node at (1.5,.8) (y) {$S_{x,1}\times S_{x,1}^{+}$};
        \node at (1,1.6) (z) {$B^{+}$};
        \draw[->] (z) -- (x) node[above,pos=.5] {$\alpha^{+}$};
        \draw[->] (z) -- (y);
        \draw[->] (x) -- (a) node[left,pos=.3] {$d_{1}$};
        \draw[->] (x) -- (b) node[right,pos=.1] {$(d_{2},d_{0})$};
        \draw[->] (y) -- (b) node[right,pos=.8] {$id\times j_{x}^{+}$};
        \draw[->] (y) -- (c);
    }\right )
$$
and
$$
\zeta_{X,x}*L_{x}(\pi)^{-} = \left ( \tikz[xscale=3,yscale=1.7,baseline=30]{
        \node at (0,0) (a) {$S_{x,1}$};
        \node at (1,0) (b) {$S_{x,1}\times S_{x,1}$};
        \node at (2,0) (c) {$*$};
        \node at (.5,.8) (x) {$S_{x,2}$};
        \node at (1.5,.8) (y) {$S_{x,1}\times S_{x,1}^{-}$};
        \node at (1,1.6) (z) {$B^{-}$};
        \draw[->] (z) -- (x) node[above,pos=.5] {$\alpha^{-}$};
        \draw[->] (z) -- (y);
        \draw[->] (x) -- (a) node[left,pos=.3] {$d_{1}$};
        \draw[->] (x) -- (b) node[right,pos=.1] {$(d_{2},d_{0})$};
        \draw[->] (y) -- (b) node[right,pos=.8] {$id\times j_{x}^{-}$};
        \draw[->] (y) -- (c);
    }\right ). 
$$
From here, the proof divides into cases based on the splitting type of the point $x$. 

First assume $x$ is a branch point of the covering $\pi$ and let $y$ be its (necessarily unique) preimage. In this case, $S_{x,1}$ has a unique initial element $\eta$, e.g.~when $X$ is an algebraic curve, $\eta$ is the trivial $0$-cycle $0 = 0x$, which we will write multiplicatively here: $\eta = x^{0}$. Take $j_{x}^{+} : S_{x,1}^{+}\hookrightarrow S_{x,1}$ to be the inclusion of the subset $\{\eta\}$. Also let $S_{x,1}^{-} = \varnothing$. Then formula (\ref{eq:local}) becomes 
$$
\pi_{*}\zeta_{Y,x} \cong \zeta_{X,x}*L_{x}(\pi)^{+}
$$
which is verified by constructing an equivalence of spans 
\begin{center}
\begin{tikzpicture}[scale=1.3]
  \node at (-1,0) (a) {$S_{x,1}$};
  \node at (0,1) (b) {$B^{+}$};
  \node at (0,-1) (c) {$T_{x,1}$};
  \node at (1,0) (d) {$*$};
  \draw[->] (b) -- (a) node[left,pos=.3] {$d_{1}\circ\alpha^{+}$};
  \draw[->] (b) -- (d);
  \draw[->] (c) -- (a) node[left,pos=.4] {$\pi_{1}$};
  \draw[->] (c) -- (d);
  \draw[->] (c) -- (b) node[right,pos=.5] {$\varphi$};
\end{tikzpicture}. 
\end{center}
Note that $B^{+}$ can be described explicitly: 
$$
B^{+} = \{\sigma\in S_{x,2} \mid d_{0}\sigma = \eta\}. 
$$
Define $\varphi$ in the above diagram by 
$$
\varphi(y^{k}) = \tikz[baseline=5]{
    \draw[thick] (-.7,0) -- (0,1) node[left,pos=.6] {$x^{k}$};
    \draw[thick] (0,1) -- (.7,0) node[right,pos=.4] {$\eta$};
    \draw[thick] (.7,0) -- (-.7,0) node[below,pos=.5] {$x^{k}$};
    \node at (0,.4) {\large $\sigma$};
}.
$$
Since every $2$-simplex in $B^{+}$ has this form, $\varphi$ is invertible (send such a $\sigma$ back to $y^{k}\in T_{x,1}$). The diagram commutes by construction, proving formula (\ref{eq:local}) in the ramified case. 

Next, suppose $x$ splits in the covering $\pi$, with fibre $\pi^{-1}(x) = \{y,\bar{y}\}$. As in the ramified case, we take $S_{x,1}^{-} = \varnothing$, but set $S_{x,1}^{+} = S_{x,1}$ so that $B^{+} = S_{x,2}$ and $d_{1}\circ\alpha = d_{1}$. To compare the linear functors $\pi_{*}\zeta_{Y,x}$ and $\zeta_{X,x}*L_{x}(\pi)^{+}$, we once again construct a map $\varphi : T_{x,1}\rightarrow B^{+}$ that makes the appropriate diagram commute. Define $\varphi$ by 
$$
\varphi(ay + b\bar{y}) = \tikz[baseline=5]{
    \draw[thick] (-.7,0) -- (0,1) node[left,pos=.6] {$x^{a}$};
    \draw[thick] (0,1) -- (.7,0) node[right,pos=.4] {$x^{b}$};
    \draw[thick] (.7,0) -- (-.7,0) node[below,pos=.5] {$x^{a + b}$};
    \node at (0,.4) {\large $\sigma$};
}.
$$
Every $\sigma\in B^{+} = S_{x,2}$ has this form, so sending $\sigma$ with $d_{2}\sigma = x^{a}$ and $d_{0}\sigma = x^{b}$ to $y^{a}\bar{y}^{b}\in T_{x,1}$ gives an inverse to $\varphi$. In addition, 
$$
d_{1}\circ\varphi(y^{a}\bar{y}^{b}) = d_{1}\left (\tikz[baseline=5]{
    \draw[thick] (-.7,0) -- (0,1) node[left,pos=.6] {$x^{a}$};
    \draw[thick] (0,1) -- (.7,0) node[right,pos=.4] {$x^{b}$};
    \draw[thick] (.7,0) -- (-.7,0) node[below,pos=.5] {$x^{a + b}$};
    \node at (0,.4) {\large $\sigma$};
}\right ) = x^{a + b} = \pi_{1}(y^{a}\bar{y}^{b})
$$
so the diagram commutes, proving formula (\ref{eq:local}) in the split case. 

Finally, suppose $x$ is inert, which doesn't occur in the case of a CW complex. Here, we take $S_{x,1}^{+}$ and $S_{x,1}^{-}$ to be the even and odd degree divisors in $S_{x,1}$, 
$$
S_{x,1}^{+} = \{x^{2k} \mid k\geq 0\} \quad\text{and}\quad S_{x,1}^{-} = \{x^{2k + 1} \mid k\geq 0\}, 
$$
together with their natural inclusions $j^{\pm} : S_{x,1}^{\pm} \hookrightarrow S_{x,1}$. Then the equivalence $\pi_{*}\zeta_{Y,x} + \zeta_{X,x}*L_{x}(\pi)^{-} \cong L_{x}(\pi)^{+}$ is verified by constructing an isomorphism $\varphi$ in the diagram 
\begin{center}
\begin{tikzpicture}[scale=1.3]
  \node at (-1,0) (a) {$S_{x,1}$};
  \node at (0,1) (b) {$B^{+}$};
  \node at (0,-1) (c) {$T_{x,1}\amalg B^{-}$};
  \node at (1,0) (d) {$*$};
  \draw[->] (b) -- (a) node[left,pos=.3] {$d_{1}\circ\alpha^{+}$};
  \draw[->] (b) -- (d);
  \draw[->] (c) -- (a) node[left,pos=.4] {$\pi_{1}\sqcup d_{1}$};
  \draw[->] (c) -- (d);
  \draw[->] (c) -- (b) node[right,pos=.5] {$\varphi$};
\end{tikzpicture}.
\end{center}
This time, $B^{+}$ and $B^{-}$ have the following descriptions: $$
B^{+} = \{\sigma\in S_{x,2} \mid d_{0}\sigma = x^{2k},k\geq 0\} \quad\text{and}\quad B^{-} = \{\sigma\in S_{x,2} \mid d_{2}\sigma = x^{2k + 1},k\geq 0\}. 
$$
We define $\varphi$ on $T_{x,1}$ by 
$$
\varphi(y^{k}) = \tikz[baseline=5]{
    \draw[thick] (-.7,0) -- (0,1) node[left,pos=.6] {$\eta$};
    \draw[thick] (0,1) -- (.7,0) node[right,pos=.4] {$x^{2k}$};
    \draw[thick] (.7,0) -- (-.7,0) node[below,pos=.5] {$x^{2k}$};
    \node at (0,.4) {\large $\sigma$};
}
$$
and on $B^{-}$ by 
$$
\varphi\left (\tikz[baseline=5]{
    \draw[thick] (-.7,0) -- (0,1) node[left,pos=.6] {$x^{k}$};
    \draw[thick] (0,1) -- (.7,0) node[right,pos=.4] {$x^{2\ell + 1}$};
    \draw[thick] (.7,0) -- (-.7,0) node[below,pos=.5] {$x^{k + 2\ell + 1}$};
    \node at (0,.4) {\large $\sigma$};
}\right ) = \tikz[baseline=5]{
    \draw[thick] (-.7,0) -- (0,1) node[left,pos=.6] {$x^{k + 1}$};
    \draw[thick] (0,1) -- (.7,0) node[right,pos=.4] {$x^{2\ell}$};
    \draw[thick] (.7,0) -- (-.7,0) node[below,pos=.5] {$x^{k + 2\ell + 1}$};
    \node at (0,.4) {\large $\tau$};
}.
$$
It is routine to check $\varphi$ is a bijection (cf.~\cite[Thm.~4.2]{ak}), which proves formula (\ref{eq:local}) in all cases. 
\end{proof}

Passing to the numerical incidence algebra $I_{\#}(S_{x})$ and pushing forward to $I_{\#}(*) \cong k[[t]]$ recovers a local version of formula (\ref{eq:genquad}). For example, when $X = \P^{1}$ and $Y = C$ is a hyperelliptic curve over $\F_{q}$, the objective formula pushes forward to 
\begin{equation}\label{eq:quadzetalocal}
Z_{x}(C,t) = Z_{x}(\P^{1},t)(L_{x}(C,t)^{+} - L_{x}(C,t)^{-}) = Z_{x}(\P^{1},t)L_{x}(C,t)
\end{equation}
where $L_{x}(C,t)$ is the local factor of the classical $L$-function of $C$, namely 
$$
L_{x}(C,t) = \frac{1}{1 - \chi(F_{x})t^{\deg(x)}}
$$
where $\chi$ is the quadratic character associated to the extension $\F_{q}(C)/\F_{q}(t)$ and $F_{x}$ is the Frobenius at $x$.

\begin{rem}
In the proof above, we chose a labeling of the points in the fibre $\pi^{-1}(x)$. Formula (\ref{eq:local}) did not depend on this choice, but the bijections constructed above are nevertheless not canonical. 
\end{rem}

\subsection{Global Theorem}
\label{sec:mainthmglob}

Next, we prove the global formula of Theorem~\ref{thm:introgen}. Keep the notation $X,Y,S,T$ and $\pi$ from above. 

\begin{thm}[{Theorem~\ref{thm:introgen}}]
\label{thm:mainthm}
In the incidence algebra $I(X)$, there is an equivalence of linear functors 
$$
\pi_{*}\zeta_{Y} + \zeta_{X}*L(\pi)^{-} \cong \zeta_{X}*L(\pi)^{+}. 
$$
\end{thm}

\begin{proof}
As before, the terms in the formula are linear functors represented by the following spans: 
\begin{align*}
    \pi_{*}\zeta_{Y} &= \left (\roofx{S_{1}}{*}{T_{1}}{\pi_{1}}{}\right )
\end{align*}
as well as
\begin{align*}
    \zeta_{X}*L(\pi)^{+} &= \left ( \tikz[xscale=3,yscale=1.7,baseline=30]{
        \node at (0,0) (a) {$S_{1}$};
        \node at (1,0) (b) {$S_{1}\times S_{1}$};
        \node at (2,0) (c) {$*$};
        \node at (.5,.8) (x) {$S_{2}$};
        \node at (1.5,.8) (y) {$S_{1}\times S_{1}^{+}$};
        \node at (1,1.6) (z) {$B^{+}$};
        \draw[->] (z) -- (x) node[above,pos=.5] {$\alpha^{+}$};
        \draw[->] (z) -- (y);
        \draw[->] (x) -- (a) node[left,pos=.3] {$d_{1}$};
        \draw[->] (x) -- (b) node[right,pos=.1] {$(d_{2},d_{0})$};
        \draw[->] (y) -- (b) node[right,pos=.8] {$id\times j^{+}$};
        \draw[->] (y) -- (c);
    }\right )\\
    \text{and}\quad \zeta_{X}*L(\pi)^{-} &= \left ( \tikz[xscale=3,yscale=1.7,baseline=30]{
        \node at (0,0) (a) {$S_{1}$};
        \node at (1,0) (b) {$S_{1}\times S_{1}$};
        \node at (2,0) (c) {$*$};
        \node at (.5,.8) (x) {$S_{2}$};
        \node at (1.5,.8) (y) {$S_{1}\times S_{1}^{-}$};
        \node at (1,1.6) (z) {$B^{-}$};
        \draw[->] (z) -- (x) node[above,pos=.5] {$\alpha^{-}$};
        \draw[->] (z) -- (y);
        \draw[->] (x) -- (a) node[left,pos=.3] {$d_{1}$};
        \draw[->] (x) -- (b) node[right,pos=.1] {$(d_{2},d_{0})$};
        \draw[->] (y) -- (b) node[right,pos=.8] {$id\times j^{-}$};
        \draw[->] (y) -- (c);
    }\right )
\end{align*}
where $S_{1}^{\pm}$ and $j^{\pm}$ are defined below. We must then construct an isomorphism $\varphi : T_{1}\amalg B^{-} \xrightarrow{\sim} B^{+}$ making the following diagram commute: 
\begin{center}
\begin{tikzpicture}[scale=1.3]
  \node at (-1,0) (a) {$S_{1}$};
  \node at (0,1) (b) {$B^{+}$};
  \node at (0,-1) (c) {$T_{1}\amalg B^{-}$};
  \node at (1,0) (d) {$*$};
  \draw[->] (b) -- (a) node[left,pos=.3] {$d_{1}\circ\alpha^{+}$};
  \draw[->] (b) -- (d);
  \draw[->] (c) -- (a) node[left,pos=.4] {$\pi_{1}$};
  \draw[->] (c) -- (d);
  \draw[->] (c) -- (b) node[right,pos=.5] {$\varphi$};
\end{tikzpicture}.
\end{center}

For a $1$-simplex $\alpha\in S_{1}$, write $\alpha = \sum_{x\in X} x^{a_{x}}$. Define $S_{1}^{+}$ and $S_{1}^{-}$ as the following subsets of $S_{1}$: 
\begin{align*}
    S_{1}^{+} &= \left\{\alpha\in S_{1} \mid a_{x} = 0 \text{ for ramified } x,\sum_{x\text{ inert}} a_{x} \text{ is even}\right\}\\
    S_{1}^{-} &= \left\{\alpha\in S_{1} \mid a_{x} = 0 \text{ for ramified and split } x,\sum_{x\text{ inert}} a_{x} \text{ is odd}\right\}. 
\end{align*}
Their natural inclusions $j^{+} : S_{1}^{+}\hookrightarrow S_{1}$ and $j^{-} : S_{1}^{-}\hookrightarrow S_{1}$ define $L(\pi)^{+}$ and $L(\pi)^{-}$ and the span compositions above. 

We now define $\varphi : T_{1}\amalg B^{-} \xrightarrow{\sim} B^{+}$. For $\varphi = \sum_{y\in|C|} y^{b_{y}}\in T_{1}$, put $\varphi(\beta) = \tau\in S_{2}$ with faces 
\begin{align*}
    d_{2}\tau &= \prod_{\substack{y\in\supp(\beta) \\ \text{ramified}}} \pi_{*}(y)^{b_{y}}\prod_{\substack{y\in\supp(\beta) \\ \text{split}}} \pi_{*}(y)^{b_{y}}\\
    d_{0}\tau &= \prod_{\substack{\bar{y}\in\supp(\beta) \\ \text{split}}} \pi_{*}(\bar{y})^{b_{\bar{y}}}\prod_{\substack{y\in\supp(\beta) \\ \text{inert}}} \pi_{*}(y)^{b_{y}}\\
    d_{1}\tau &= \pi_{*}\beta = \prod_{y\in X} \pi_{*}(y)^{b_{y}}. 
\end{align*}
Then since $d_{0}\tau$ is supported away from the ramification locus of $\pi$ and the exponent of $x$ in $\pi_{*}(y)$ is even for inert $y\mapsto x$, we see that $\tau\in B^{+}$. On the other hand, for $\sigma\in B^{-}$ with $d_{2}\sigma = \alpha,d_{0}\sigma = \beta$ and $d_{1}\sigma = \gamma = \alpha\beta$, define $\varphi(\sigma) = \tau$ with faces 
\begin{align*}
    d_{2}\tau &= \prod_{x\text{ ramified}} x^{a_{x}}\prod_{x\text{ split}} x^{a_{x}}\prod_{x\text{ inert}} x^{a_{x} + 1}\\
    d_{0}\tau &= \prod_{\substack{x\in\supp(\beta) \\ \text{inert}}} x^{b_{x} - 1} \quad\text{and}\quad d_{1}\tau = \gamma = \alpha\beta. 
\end{align*}
Then as above, $d_{0}\tau$ is supported away from the ramified points in $X$ and $\beta = d_{0}\sigma\in S_{1}^{-}$ implies $b_{x} - 1$ is even, so $\tau\in B^{+}$. 

To show $\varphi$ is an isomorphism, we construct an inverse $\varphi^{-1} : B^{+}\rightarrow T_{1}\amalg B^{-}$ as follows. If $\tau\in B^{+}$ has $d_{2}\tau = \alpha,d_{0} = \beta$ and $d_{1}\tau = \gamma = \alpha\beta$, with $\alpha$ supported away from all inert points in $X$, send it to 
$$
\varphi^{-1}(\tau) = \sum_{x\text{ ramified}} y^{c_{x}}\prod_{x\text{ split}} y^{a_{x}}\bar{y}^{b_{x}}\prod_{\substack{x\text{ inert} \\ b_{x}\text{ even}}} y^{b_{x}/2}\prod_{\substack{x\text{ inert} \\ b_{x}\text{ odd}}} y^{b_{x}} \;\in T_{1}. 
$$
If $\alpha$ is supported on any inert points, send $\tau$ to $\beta^{-1}(\tau) = \sigma\in B^{-}$ with faces 
\begin{align*}
    d_{2}\sigma &= \prod_{x\text{ ramified}} x^{a_{x}}\prod_{x\text{ split}} x^{a_{x}}\prod_{\substack{x\in\supp(\alpha) \\ \text{ inert}}} x^{a_{x} - 1}\\
    d_{0}\sigma &= \sum_{x\text{ ramified}} x^{b_{x}}\prod_{x\text{ split}} x^{b_{x}}\prod_{x\text{ inert}} x^{b_{x} + 1}\\
    d_{1}\sigma &= d_{1}\tau = \gamma. 
\end{align*}
By construction, $\varphi$ is a bijection and it satisfies $d_{1}\circ\varphi = \pi\amalg d_{1}$, completing the proof. 
\end{proof}

\begin{rem}
Recall that for each split point $x\in X$, we are choosing a labeling $\pi^{-1}(x) = \{y,\bar{y}\}$ in order to write down the formulas in the above proof. This implies that the isomorphism $\varphi$ is not canonical: it depends on a choice of section of $\pi$ over the split locus. 
\end{rem}

\begin{rem}
\label{rem:card}
Taking the cardinality of the formula in Theorem~\ref{thm:mainthm} recovers formulas (\ref{eq:quadzeta}), (\ref{eq:Lfactor}), (\ref{eq:Macdonaldsplit}) and (\ref{eq:Iharasplit}), so Theorem~\ref{thm:mainthm} can be considered a categorification of those zeta function formulas. Explicitly, cardinality is a map from the objective incidence algebra $I(X)$ to the reduced numerical incidence algebra $I_{\#}(X)$: 
\begin{align*}
    |\cdot| : I(X) &\longrightarrow I_{\#}(X)\\
    f = \left (\roofx{S_{1}}{*}{M}{v}{}\right ) &\longmapsto \biggl (|f| : \alpha \mapsto |v^{-1}(\alpha)|\biggr ). 
\end{align*}
Applying $|\cdot|$ to the formula $\pi_{*}\zeta_{Y} + \zeta_{X}*L(\pi)^{-} \cong \zeta_{X}*L(\pi)^{+}$ produces a formula which, upon further applying one of the ``generating function maps'' from Sections~\ref{sec:incalgfield}, \ref{sec:0cyc}, \ref{sec:incalgCW} and~\ref{sec:incalggraph}, recovers formulas (\ref{eq:quadzeta}), (\ref{eq:Lfactor}), (\ref{eq:Macdonaldsplit}) and (\ref{eq:Iharasplit}), respectively. 
\end{rem}

\begin{rem}
\label{rem:posetversion}
When $X$ is the ring of integers in a number field, an algebraic curve over $\F_{q}$ or a finite graph, it is possible to formulate a version of Theorem~\ref{thm:mainthm} in the ``full'' incidence algebra of the poset $S = Z_{0}^{\eff}(X)$, rather than treating it as a monoid. This corresponds to first taking the decalage of the monoid structure, then applying the same techniques as above. We have elected to omit further discussion of these formulas since they don't appear to have practical applications. Nevertheless, they provide new relations not just among effective $0$-cycles in a cover (and therefore points), but also intervals of effective $0$-cycles (and therefore families of points). 
\end{rem}

\begin{rem}
\label{rem:arbitrarydoublecovers}
The $X = C/\F_{q}$ case of Theorem~\ref{thm:mainthm} generalizes easily to arbitrary double covers of algebraic curves $C\rightarrow D$ over $\F_{q}$, giving interesting formulas for their zeta functions in $k[[t]]$: 
\begin{equation}\label{eq:arbitrarybase}
Z(C,t) = Z(D,t)L(C/D,t)
\end{equation}
for an appropriate ``relative $L$-function'' $L(C/D,t)$. The novelty of Theorem~\ref{thm:mainthm} is that it expresses this relation directly in the incidence algebra of $D$ (objectively or numerically), which is stronger than (\ref{eq:arbitrarybase}) even in the case when $D = \P^{1}$. 
\end{rem}

\subsection{General Topological Covers}

Beyond topological double covers, which are already handled by Theorem~\ref{thm:mainthm}, we can use objective linear algebra to categorify formula (\ref{eq:Macdonaldsplit}) for ramified covers of any degree. 

First, recall formula (\ref{eq:topEuler}), which says that for an unbranched cover $Y\rightarrow X$ of degree $n$, $\chi(Y) = n\chi(X)$. Let's lift this formula to an objective formula in an incidence algebra. 

As in Section~\ref{sec:incalgCW}, let $S_{\bullet}(X)$ (resp.~$S_{\bullet}(Y)$) be the simplicial complex associated to $X$ (resp.~$Y$) and let $S(\pi) : S_{\bullet}(Y)\rightarrow S_{\bullet}(X)$ be the simplicial map induced by $\pi$. This induces a map of incidence algebras $\pi_{*} : I(Y)\rightarrow I(X)$. 

In dimension $0$, a degree $n$ cover $Y\rightarrow *$ is just a set of $n$ points $Y = \{y_{1},\ldots,y_{n}\}$, so formula (\ref{eq:topEuler}) in this case is the decategorification of 
$$
\zeta_{Y} \cong \underbrace{\zeta_{*}*\cdots*\zeta_{*}}_{n}. 
$$
This works because points have no nondegenerate simplices of higher dimension. 

Unfortunately, this only applies to $0$-simplices. The M\"{o}bius function is better suited to computing Euler characteristic objectively. For any $X$, the zeta functor in $I(S_{\bullet}(X))$ is 
$$
\zeta_{X} = \left (
\roofx{S_{1}(X)}{*}{S_{1}(X)}{id}{}
\right ). 
$$
Define $\zeta_{X}^{(k)}\in I(S_{\bullet}(X))$ by 
$$
\zeta_{X}^{(k)} = \left (\roofx{S_{1}(X)}{*}{S_{k}(X)^{\circ}}{d_{1}^{k}}{}\right )
$$
where $S_{k}(X)^{\circ}\subseteq S_{k}(X)$ is the subset of nondegenerate $k$-simplices. (This is not to be confused with the $k$-fold convolution $\zeta_{X}^{k}$.) Following \cite{gkt2}, we also put $\zeta_{X}^{(0)} = \delta_{X}$, the unit of convolution in $I(S_{\bullet}(X))$, and define 
$$
\Phi_{even} = \sum_{k = 0}^{\infty} \zeta_{X}^{(2k)} \quad\text{and}\quad \Phi_{odd} = \sum_{k = 0}^{\infty} \zeta_{X}^{(2k + 1)}. 
$$
These functors satisfy an objective form of M\"{o}bius inversion \cite[Thm.~3.8]{gkt2}: 
$$
\zeta_{X}*\Phi_{even} \cong \delta + \zeta_{X}*\Phi_{odd}. 
$$
In the numerical incidence algebra $I_{\#}(S_{\bullet}(X))$, this formula descends to 
$$
\delta = \zeta_{X}*(\Phi_{even} - \Phi_{odd}) = \zeta_{X}*\mu_{X}
$$
where $\mu_{X}$ is the numerical M\"{o}bius function for $S_{\bullet}(X)$. 

Define spaces $E_{even}(X),E_{odd}(X)\in\cat{Top}$ by 
$$
E_{even}(X) = \coprod_{s\in S_{1}(X)} \Phi_{even}(s) \quad\text{and}\quad E_{odd}(X) = \coprod_{s\in S_{1}(X)} \Phi_{odd}(s). 
$$
Here, $F(s)$ denotes the fibre $M_{s} = p^{-1}(s)$ for any linear functor $F$ represented by a span 
$$
F = \left (\roofx{S_{1}(X)}{*}{M}{p}{}\right ). 
$$
Already, this categorifies the Euler characteristic: 

\begin{prop}
If $X$ is a finite CW complex, then $\chi(X) = |E_{even}(X)| - |E_{odd}(X)|$. 
\end{prop}

We can go beyond Euler characteristic. The objects $E_{even}(X)$ and $E_{odd}(X)$ can be viewed as linear functors in $I(S_{\bullet}(X))$: 
$$
\Phi_{even}(X) = \left (\roof{S_{1}(X)}{*}{E_{even}(X)}\right ) \quad\text{and}\quad \Phi_{odd}(X) = \left (\roof{S_{1}(X)}{*}{E_{odd}(X)}\right ). 
$$
In this way, the construction can be interpreted as an ``objective Macdonald polynomial'': pushing forward $\Phi_{even}(X)$ and $\Phi_{odd}(X)$ along the terminal map $t : X\rightarrow *$ produces linear functors $t_{*}\Phi_{even}(X),t_{*}\Phi_{odd}(X)\in I(S_{\bullet}(*))$ and, as discussed above, their decategorifications in $I_{\#}(S_{\bullet}(*)) \cong k[[t]]$ are 
$$
|t_{*}\Phi_{even}(X)| = (1 - t)^{-|E_{even}(X)|} \quad\text{and}\quad |t_{*}\Phi_{odd}(X)| = (1 - t)^{-|E_{odd}(X)|}
$$
so that 
$$
\frac{|t_{*}\Phi_{even}(X)|}{|t_{*}\Phi_{odd}(X)|} = (1 - t)^{-|E_{even}(X)| + |E_{odd}(X)|} = (1 - t)^{-\chi(X)}. 
$$

\begin{thm}
For a degree $n$ cover $\pi : Y\rightarrow X$ of finite CW complexes, there is an equivalence of linear functors 
$$
\pi_{*}\Phi_{even}(Y)*\Phi_{odd}(X)^{n} \cong \pi_{*}\Phi_{odd}(Y)*\Phi_{even}(X)^{n}
$$
in the objective incidence algebra $I(S_{\bullet}(X))$. Here, $(-)^{n}$ denotes the $n$-fold convolution. 
\end{thm}

\begin{proof}
Each nondegenerate $k$-cell in $X$ lifts to exactly $n$ nondegenerate $k$-cells in $Y$. This gives us bijections $E_{even}(Y)\xrightarrow{\sim} \coprod_{i = 1}^{n} E_{even}(X)$ and $E_{odd}(Y)\xrightarrow{\sim} \coprod_{i = 1}^{n} E_{odd}(X)$ which assemble into the stated formula. 
\end{proof}

Pushing forward to the incidence algebra of a point and taking cardinalities recovers formula (\ref{eq:topEuler}) via generating functions: 
$$
(1 - t)^{-\chi(Y)} = \frac{|t_{*}\Phi_{even}(Y)|}{|t_{*}\Phi_{odd}(Y)|} = \left (\frac{|t_{*}\Phi_{even}(X)|}{|t_{*}\Phi_{odd}(X)|}\right )^{n} = (1 - t)^{-n\chi(X)}. 
$$

\begin{thm}
\label{thm:top}
For a connected, degree $n$ branched cover $\pi : Y\rightarrow X$, there is a linear functor $R\in I(S_{\bullet}(Y))$ and an equivalence 
$$
\pi_{*}\Phi_{even}(Y)*\Phi_{odd}(X)^{n}*\pi_{*}R \cong \pi_{*}\Phi_{odd}(Y)*\Phi_{even}(X)^{n}
$$
in $I(S_{\bullet}(X))$. 
\end{thm}

\begin{proof}
For simplicity, we give the proof in the case when $\pi : Y\rightarrow X$ is a branched cover of surfaces. Let $X_{br} = \{x_{1},\ldots,x_{r}\}$ be the branch locus of $\pi$ and for each $x_{i}$, write $\pi^{-1}(x_{i}) = \{y_{i1},\ldots,y_{ik_{i}}\}$ for each $1\leq i\leq r$. Each nondegenerate $k$-cell $\sigma\in S_{k}(X)$ lifts to $n$ nondegenerate $k$-cells in $S_{k}(Y)$ \emph{except} when when $k = 0$ and $\sigma = x_{i}$ for some $1\leq i\leq r$. In this case, there are $n - k_{i} = \sum_{j = 1}^{k_{i}} (e(y_{ij}) - 1)$ points ``missing'' from the count. Said another way, there is a bijection 
$$
\coprod_{i = 1}^{n} E_{even}(X) \xrightarrow{\sim} E_{even}(Y)\amalg\{z_{i1},\ldots,z_{i,n - k_{i}}\}_{i = 1}^{r}
$$
where $z_{i1},\ldots,z_{i,n - k_{i}}$ are these ``missing points'' in $\pi^{-1}(x_{i})$. Setting $R = \{z_{i1},\ldots,z_{i,n - k_{i}}\}_{i = 1}^{r}$, we obtain the stated formula. 
\end{proof}

\begin{rem}
Let $G$ be the group of deck transformations of $\pi : Y\rightarrow X$. Then it should be possible to lift formula (\ref{eq:branchedEuler}) to an equivalence of linear functors of $G$-representations, eliminating the need to separate $\Phi_{even}$ and $\Phi_{odd}$ in the above formulation. See Appendix~\ref{app:Grep}. 
\end{rem}

\begin{question}
Is it also possible to prove $\chi(X\times Y) = \chi(X)\chi(Y)$ this way? And more generally $\chi(E) = \chi(B)\chi(F)$ for a fibration $F\rightarrow E\rightarrow B$?
\end{question}

\subsection{A Motivic Formula}
\label{sec:motivic}

In the case of a cover of curves $Y\rightarrow X$ over a field $k$, Theorem~\ref{thm:mainthm} can be generalized to a motivic formula by replacing the decomposition sets $Z_{0}^{\eff}(X)$ and $Z_{0}^{\eff}(Y)$ with the {\it decomposition schemes} $S^{\bullet}X$ and $S^{\bullet}Y$, where 
$$
S^{n}X = \coprod_{(j_{0},\ldots,j_{n - 1})\in\N_{0}^{n}} \prod_{i = 0}^{n - 1} \Sym^{j_{i}}X
$$
with the simplicial scheme structure induced by the simplicial structure on the indices (i.e.~the additive monoid $\N_{0}$). That is, $S^{\bullet}X$ is the decomposition scheme attached to the monoid scheme $\coprod_{n\geq 0} \Sym^{n}X$. A variant of the proof of Theorem~\ref{thm:mainthm} gives the following formula. 

\begin{thm}
\label{thm:motivicformula}
Let $\pi : Y\rightarrow X$ be a double cover of varieties over a field $k$ and let $I_{mot}(X)$ be the incidence algebra of the decomposition scheme $S^{\bullet}X$. Then in $I_{mot}(X)$, there is an equivalence of linear functors 
$$
\pi_{*}\zeta_{Y} + \zeta_{X}*L(\pi)^{-} \cong \zeta_{X}*L(\pi)^{+}. 
$$
\end{thm}

\begin{proof}
Viewing the scheme 
$$
S^{1}X = \coprod_{n = 0}^{\infty} \Sym^{n}X
$$
via its functor of points, we can extend the definitions of $S_{1}^{+}$ and $S_{1}^{-}$ from Theorem~\ref{thm:mainthm} functorially to subschemes $(S^{1}X)^{+},(S^{1}X)^{-}\subseteq S^{1}X$ which determine spans $L(\pi)^{+}$ and $L(\pi)^{-}$ representing linear functors in the incidence algebra $I(X)$. The rest of the proof goes through as before, by the Yoneda lemma. 
\end{proof}

This in turn determines an analogous formula in $R$ for each motivic measure $\mu : K_{0}(\cat{Var}_{k})\rightarrow R$, where $K_{0}(\cat{Var}_{k})$ is the Grothendieck ring of $k$-varieties. Several examples include: 
\begin{itemize}
    \item $k = \F_{q}$, $R = \Z$, $\mu = \#(-)(\F_{q})$, the point-counting measure over the finite field $\F_{q}$. Since $\Sym^{n}X(\F_{q}) = Z_{0}^{\eff}(X)_{n}$, the set of effective $0$-cycles of degree $n$, this decategorification recovers Theorem~\ref{thm:mainthm}. 
    \item $k = \C$, $R = \Z$, $\mu = $ Euler characteristic, which recovers formula (\ref{eq:Macdonaldsplit}) for Riemann surfaces. 
    \item $R = \Z[SB]$, the ring of stable birational equivalence classes, and $\mu = $ the Larsen--Lunts measure \cite{ll}. 
    \item $R = K_{0}(\cat{Var}_{k})$, $\mu = \operatorname{id}$, giving a universal version of (\ref{eq:genquad}) in $K_{0}(\cat{Var}_{k})$. Also compare this to the approach in \cite{dh}. 
\end{itemize}

\section{Applications}
\label{sec:app}

\subsection{Point Counts on Elliptic Curves}
\label{sec:pointcounts}

Let $E$ be an elliptic curve over $\F_{q}$ with $L$-polynomial $L(E,t) = 1 - a_{q}(E)t + qt^{2}$ where $a_{q}(E) = q + 1 - \#E(\F_{q})$. On one hand, formula (\ref{eq:Lfactor}) says that 
$$
Z(E,t) = Z(\P^{1},t)L(E,t). 
$$
On the other hand, Theorems~\ref{thm:mainthmloc} and~\ref{thm:mainthm}, through Remark~\ref{rem:card}, show that $L(E,t) = L(E,t)^{+} - L(E,t)^{-}$ can be written as 
$$
L(E,t) = \prod_{x\in|\P^{1}|} (1 - \chi(x)t^{\deg(x)})^{-1}
$$
where $\chi$ is the following quadratic character: 
$$
\chi(x) = \begin{cases}
    0, & x \text{ is ramified}\\
    1, & x \text{ is split}\\
    -1, & x \text{ is inert}
\end{cases}
$$
which can be extended multiplicatively to all effective $0$-cycles. 

Let's spell this out explicitly. The $t^{n}$ coefficient of $L(E,t)^{+}$ is the cardinality of the fibre of $j^{+} : S_{1}^{+}\rightarrow S_{1}$ over the set $S_{1}(n)$ of $0$-cycles of degree $n$ in $S_{1}$: 
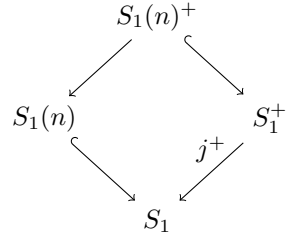
\begin{center}
\begin{tikzpicture}[xscale=3,yscale=1.7,baseline=30]
    \node at (1,0) (a) {$S_{1}$};
    \node at (.5,.8) (b) {$S_{1}(n)$};
    \node at (1.5,.8) (c) {$S_{1}^{+}$};
    \node at (1,1.6) (d) {$S_{1}(n)^{+}$};
    \draw[->] (d) -- (b);
    \draw[right hook ->] (d) -- (c);
    \draw[right hook ->] (b) -- (a);
    \draw[->] (c) -- (a) node[above,pos=.5] {$j^{+}$};
\end{tikzpicture}.
\end{center}
Then $\#S_{1}(n)^{+}$ is equal to the number of effective $0$-cycles $\alpha = \sum_{x} a_{x}x$ on $\P^{1}$ of degree $n$ with $a_{x} = 0$ for ramified $x$ and the sum of the $a_{x}$ even for inert $x$; $\#S_{1}(n)^{-}$ is defined similarly. Putting these together, 
$$
\#S_{1}(n)^{+} - \#S_{1}(n)^{-} = \sum_{\deg(\alpha) = n} \chi(\alpha). 
$$
This allows us to reinterpret the $L$-polynomial in terms of the character $\chi$, as claimed. Write 
$$
\prod_{x\in|\P^{1}|} (1 - \chi(x)t^{\deg(x)})^{-1} = \sum_{n = 0}^{\infty} f(n)t^{n} \quad\text{where}\quad f(n) = \sum_{\deg(\alpha) = n} \chi(\alpha). 
$$
As an element of the numerical incidence algebra $I_{\#}(Z_{0}^{\eff}(\P^{1}))$, $f$ is the sequence 
$$
f(0) = 1, \quad f(1) = -a_{q}(E), \quad f(2) = q \quad\text{and}\quad f(n) = 0 \text{ for } n > 2.
$$
For $n = 1$, this implies the first part of formula (\ref{eq:aq}): 
\begin{equation}\label{eq:aq1}
a_{q}(E) = -\sum_{x\in\P^{1}(\F_{q})} \chi(x) = \#\{\text{inert } x\in\P^{1}(\F_{q})\} - \#\{\text{split } x\in\P^{1}(\F_{q})\} =: i_{q}(E) - s_{q}(E)
\end{equation}
while for $n = 2$, we obtain 
$$
q = \sum_{\deg(\alpha) = 2} \chi(\alpha). 
$$
Formula (\ref{eq:Lfactor}) then implies 
\begin{equation}\label{eq:aq2}
\#\Sym^{2}E(\F_{q}) = (q^{2} + q + 1) - (q + 1)a_{q} + \sum_{\deg(\alpha) = 2} \chi(\alpha). 
\end{equation}
A similar proof shows the rest of formula (\ref{eq:aq}): 
\begin{equation}
\#\Sym^{n}E(\F_{q}) = \sum_{i + j = n} (q^{i} + \ldots + q + 1)\sum_{\deg(\alpha) = j} \chi(\alpha). 
\end{equation}

These are the recursions which generate all point counts $\#E(\F_{q^{r}})$ from knowledge of just $\#E(\F_{q})$. (To go from the sequence $\#\Sym^{n}E(\F_{q})$ to the sequence $\#E(\F_{q^{n}})$, take ghost components, or equivalently, apply the logarithmic derivative to $Z(E,t)$.) Our proof demonstrates that they come directly from an objective $L$-function and, more importantly, that this $L$-function enumerates points (weighted by $\chi$) \emph{on the projective line}. 

\begin{ex}
\label{ex:1.5.d.f5}
Consider the elliptic curve $E/\F_{5}$ \cite[\href{https://www.lmfdb.org/Variety/Abelian/Fq/1/5/d}{Abelian Variety 1.5.d}]{lmfdb} with Weierstrass equation 
$$
E : y^{2} = x^{3} + x + 1. 
$$
It is known that $\#E(\F_{5}) = 9$ (see below), giving $L(E,t) = 1 + 3t + 5t^{2}$. Projecting to $\P^{1}$, there are $3$ more split points than inert points. Since $x^{3} + x + 1$ is irreducible over $\F_{5}$, $\infty$ is the only ramified point on $\P^{1}$. Split points contribute $2$ rational points on $E$, so there must be exactly $4$ split points and $1$ inert point downstairs. Here's a table of the closed points of $E$ with images in $\P^{1}(\F_{5})$, organized by splitting type of their images. 
\begin{center}
\begin{tabular}{|c|c|c|c|c|c|}
    \hline
    point & ramified image & point & split image & point & inert image\\
    \hline
    $O$ & $\infty$ & $(0,1)$ & $0$  & $\left (1,\sqrt{3}\right )$   & $1$ \\
         &          & $(0,4)$ &     & $\left (1,-\sqrt{3}\right )$  & \\
         &          & $(2,1)$ & $2$ &                               & \\
         &          & $(2,4)$ &     &                               & \\
         &          & $(3,1)$ & $3$ &                               & \\
         &          & $(3,4)$ &     &                               & \\
         &          & $(4,2)$ & $4$ &                               & \\
         &          & $(4,3)$ &     &                               & \\
    \hline
\end{tabular}
\end{center}
The $8$ split points cover the $4$ split points on $\P^{1}$ identified above. The inert points $\left (1,\pm\sqrt{3}\right )$ cover the inert point $1\in\P^{1}(\F_{5})$, with $\sqrt{3}$ a fixed square root of $3$ in $\F_{25}$. Moreover, formula (\ref{eq:aq2}) quickly allows us to compute that $\#E(\F_{25}) = 27$. We can further deduce then that in $\P^{1}(\F_{25})$, there are $12$ split points, lifting to $24$ of the $27$ points in $E(\F_{25})$, and $13$ inert points. Similarly, $\#E(\F_{125}) = 108$ and the $126$ points in $\P^{1}(\F_{125})$ consist of $4$ ramified points, $31$ split points and $91$ inert points. 
\end{ex}

\begin{ex}
\label{ex:1.5.ad.f5}
Up to isogeny, the elliptic curve $E$ in Example~\ref{ex:1.5.d.f5} has a unique quadratic twist over $\F_{5}$ \cite[\href{https://www.lmfdb.org/Variety/Abelian/Fq/1/5/ad}{Abelian Variety 1.5.ad}]{lmfdb}, 
$$
E' : y^{2} = x^{3} + 4x + 2
$$
with $\#E'(\F_{5}) = 3$ and $L(E',t) = 1 - 3t + 5t^{2}$. This time, there are $3$ more inert points than split points, from which we can deduce that on $\P^{1}$, there is a single ramification point at $\infty$, $4$ inert points and $1$ split point. Here's the corresponding table of points for $E'$: 
\begin{center}
\begin{tabular}{|c|c|c|c|c|c|}
    \hline
    point & ramified image & point & split image & point & inert image\\
    \hline
    $O$ & $\infty$ & $(3,1)$ & $3$  & $\left (0,\sqrt{2}\right )$   & $0$ \\
         &          & $(3,4)$ &     & $\left (0,-\sqrt{2}\right )$  & \\
         &          &         &     & $\left (1, \sqrt{2}\right )$  & $1$ \\
         &          &         &     & $\left (1,-\sqrt{2}\right )$  & \\
         &          &         &     & $\left (2,\sqrt{3}\right )$   & $2$ \\
         &          &         &     & $\left (2,-\sqrt{3}\right )$  & \\
         &          &         &     & $\left (4,\sqrt{2}\right )$   & $4$ \\
         &          &         &     & $\left (4,-\sqrt{2}\right )$  & \\
    \hline
\end{tabular}
\end{center}
Using formula (\ref{eq:aq2}), we quickly deduce that $\#E'(\F_{25}) = 27$, which is as expected since $E'\cong E$ over $\F_{25}$. 
\end{ex}

An interesting consequence of our interpretation of $a_{q}(E)$ is that $E$ is supersingular over $\F_{q}$ if and only if the $\F_{q}$-points of $E$ project to an equal number of split and inert points in $\P^{1}$. Here is an example of such a curve. 

\begin{ex}
\label{ex:1.7.a.f7}
Consider the following supersingular elliptic curve over $\F_{7}$ \cite[\href{https://www.lmfdb.org/Variety/Abelian/Fq/1/7/a}{Abelian Variety 1.7.a}]{lmfdb}: 
$$
E : y^{2} = x^{3} + 5x. 
$$
Then $\#E(\F_{7}) = 8$ and $L(E,t) = 1 + 7t^{2}$, so formula (\ref{eq:aq1}) says there are an equal number of split and inert points on $\P^{1}$. With $4$ ramified points downstairs, this forces $2$ split points and $2$ inert points: 
\begin{center}
\begin{tabular}{|c|c|c|c|c|c|}
    \hline
    point & ramified image & point & split image & point & inert image\\
    \hline
    $O$     & $\infty$  & $(2,2)$ & $2$     & $\left (1,\sqrt{6}\right )$ & $1$\\
    $(0,0)$ & $0$       & $(2,5)$ &         & $\left (1,-\sqrt{6}\right )$ & \\
    $(3,0)$ & $3$       & $(6,1)$ & $6$     & $\left (5,\sqrt{3}\right )$ & $5$\\
    $(4,0)$ & $4$       & $(6,6)$ &         & $\left (5,-\sqrt{3}\right )$ & \\
    \hline
\end{tabular}
\end{center}
Also, formula (\ref{eq:aq2}) gives us $\#E(\F_{49}) = 64$. 
\end{ex}

\begin{ex}
\label{ex:1.37.am.f37}
Consider the same elliptic curve $E : y^{2} = x^{3} + 5x$ over $\F_{37}$, where it is no longer supersingular \cite[\href{https://www.lmfdb.org/Variety/Abelian/Fq/1/37/am}{Abelian Variety 1.37.am}]{lmfdb}. This time, $\#E(\F_{37}) = 26$ and $L(E,t) = 1 - 12t + 37t^{2}$, so formula (\ref{eq:aq1}) implies there are $12$ more inert points than split points on $\P^{1}$. With $2$ ramified points downstairs, there must be $12$ split points and therefore $24$ inert points. (The split points have $x$-coordinates $4,8,9,12,15,17,20,22,25,28,29$ and $33$.) Formula (\ref{eq:aq2}) also confirms that $\#E(\F_{1369}) = 1300$.
\end{ex}

In this section, we showed how Theorem~\ref{thm:mainthm} gives an objective interpretation of the coefficients of the $L$-polynomial $L(E,t)$. Famously, these coefficients also have a cohomological origin, namely $L(E,t) = \det(1 - Ft\mid H^{1}(E))$ where $H^{1}(E,\Q_{\ell})$ is the $1$st $\ell$-adic cohomology group of $E$ and $F$ is the Frobenius operator. In Appendix~\ref{app:Lpoly}, we explain how these cohomological coefficients can be interpreted objectively as well, though not in the category of sets as was done in Theorem~\ref{thm:mainthm}.

\subsection{Prime Counts on Graphs}
\label{sec:loopcounts}

In a similar vein, we can use formula (\ref{eq:Iharasplit}) to deduce the zeta function of a complicated graph $Y$ from information about a simpler graph $X$ to which $Y$ admits a double cover $\pi : Y\rightarrow X$. The proofs of Theorems~\ref{thm:mainthmloc} and~\ref{thm:mainthm} show that the Artin--Ihara $L$-function $L(\pi,u) = L(\chi,u)$ can be written 
$$
L(\pi,u) = \sum_{n = 0}^{\infty} f(n)u^{n}
$$
where 
$$
f(n) = \#S_{1}(n)^{+} - \#S_{1}(n)^{-} = \sum_{\len(\alpha) = n} \chi(\alpha). 
$$
The quadratic character $\chi$ is defined on prime loops $\gamma$ by 
$$
\chi(\gamma) = \begin{cases}
    1, & \gamma \text{ is split}\\
    -1, & \gamma \text{ is inert}.
\end{cases}
$$
For example, for $n = 1$, we recover the formula 
\begin{equation}
f(1) = \#\{\text{primes } \gamma \mid \len(\gamma) = 1,\gamma \text{ is split}\} -  \#\{\text{primes } \gamma \mid \len(\gamma) = 1,\gamma \text{ is inert}\}
\end{equation}
in analogy with formula (\ref{eq:aq1}) for $a_{q}(E)$. The higher length formulas allow one to compute the coefficients of $\zeta_{Y}(u)$ recursively, as shown below. 

\begin{ex}
\label{ex:loopcount}
Let $X = B_{2}$ be the bouquet of two loops, pictured below. 
\begin{center}
\begin{tikzpicture}[decoration={markings,mark=at position 0.5 with {\arrow[scale=2]{>}}}]
    \fill (0,0) circle (.07);
    \draw[postaction=decorate] (0,0) arc(180:540:.5) node[right,pos=.5] {$a$};
    \draw[postaction=decorate] (0,0) arc(180:540:1) node[right,pos=.5] {$b$};
    \node at (-1,0) {$X$};
\end{tikzpicture}
\end{center}
The edge adjacency matrix is 
$$
T_{X} = \begin{pmatrix}
    1 & 0 & 1 & 1\\
    0 & 1 & 1 & 1\\
    1 & 1 & 1 & 0\\
    1 & 1 & 0 & 1
\end{pmatrix}
$$
which has reciprocal characteristic polynomial 
$$
\zeta_{X}(u) = \det(1 - uT_{X})^{-1} = \frac{1}{1 - 4u + 2u^{2} + 4u^{3} - 3u^{4}} = 1 + 4u + 14u^{2} + 44u^{3} + 135u^{4} + \ldots
$$
So for example, $X$ has 
\begin{itemize}
    \item $4$ prime loops of length $1$: $a,a^{\iota},b,b^{\iota}$, 
    \item $14$ effective $0$-cycles of length $2$: $ab,ab^{\iota},a^{\iota}b,a^{\iota}b^{\iota},2a,a + a^{\iota},a + b,a + b^{\iota},2a^{\iota},a^{\iota} + b,a^{\iota} + b^{\iota},2b,b + b^{\iota},2b^{\iota}$. 
\end{itemize}
There is a Galois double cover $\pi : Y\rightarrow X$ of this by the graph 
\begin{center}
\begin{tikzpicture}[decoration={markings,mark=at position 0.5 with {\arrow[scale=2]{>}}}]
    \fill (0,0) circle (.07);
    \fill (1,0) circle (.07);
    \draw[postaction=decorate] (0,0) arc(0:360:.5) node[left,pos=.5] {$b_{1}$};
    \draw[postaction=decorate] (0,0) to[out=210,in=330] (1,0);
        \node at (.5,.5) {$a_{1}$};
    \draw[postaction=decorate] (1,0) to[out=150,in=30] (0,0);
        \node at (.5,-.5) {$a_{2}$};
    \draw[postaction=decorate] (1,0) arc(-180:180:.5) node[right,pos=.5] {$b_{2}$};
    \node at (-2.5,0) {$Y$};
\end{tikzpicture}
\end{center}
with $\pi(a_{1}) = \pi(a_{2}) = a$ and $\pi(b_{1}) = \pi(b_{2}) = b$, and likewise for the conjugate edges. Its edge adjacency matrix is 
$$
T_{Y} = \begin{pmatrix}
    0 & 0 & 1 & 0 & 1 & 1 & 0 & 0\\
    0 & 0 & 0 & 1 & 0 & 0 & 1 & 1\\
    1 & 0 & 0 & 0 & 0 & 0 & 1 & 1\\
    0 & 1 & 0 & 0 & 1 & 1 & 0 & 0\\
    0 & 1 & 1 & 0 & 1 & 0 & 0 & 0\\
    0 & 1 & 1 & 0 & 0 & 1 & 0 & 0\\
    1 & 0 & 0 & 1 & 0 & 0 & 1 & 0\\
    1 & 0 & 0 & 1 & 0 & 0 & 0 & 1
\end{pmatrix}
$$
so its zeta function is 
\begin{align*}
    \zeta_{Y}(u) &= \det(1 - uT_{Y})^{-1} = \frac{1}{1 - 4u + 4u^{2} - 4u^{3} - 2u^{4} + 20u^{5} - 12u^{6} - 12u^{7} + 9u^{8}}\\
        &= 1 + 4u + 12u^{2} + 36u^{3} + 114u^{4} + \ldots
\end{align*}

Instead of this unwieldy computation, we can use formula (\ref{eq:Iharasplit}) to simplify things. As above, write 
$$
L(\pi,u) = \sum_{n = 0}^{\infty} f(n)u^{n}
$$
where $f(n) = \#S_{1}(n)^{+} - \#S_{1}(n)^{-}$. In this example, $S_{1}(n)^{+}$ (resp.~$S_{1}(n)^{-}$) is the set of effective $0$-cycles of length $n$ on $X$ containing an even (resp.~odd) number of copies of $a$ and $a^{\iota}$. For instance, $S_{1}(1)^{+} = \{b,b^{\iota}\}$ and $S_{1}(1)^{-} = \{a,a^{\iota}\}$, so $f(1) = 2 - 2 = 0$. Likewise, 
$$
S_{1}(2)^{+} = \{2a,a + a^{\iota},2a^{\iota},2b,b + b^{\iota},2b^{\iota}\} \quad\text{and}\quad S_{1}(2)^{-} = \{ab,ab^{\iota},a^{\iota}b,a^{\iota}b^{\iota},a + b,a + b^{\iota},a^{\iota} + b,a^{\iota} + b^{\iota}\}
$$
so $f(2) = 6 - 8 = -2$. More computations show that $f(3) = 0$ and $f(4) = 7$, and since $L(\pi,u)^{-1}$ must be a polynomial of degree $4$, we can deduce that $L(\pi,u)^{-1} = 1 + 2u^{2} - 3u^{4}$ and thus 
$$
L(\pi,u) = \frac{1}{1 + 2u^{2} - 3u^{4}} = 1 - 2u^{2} + 7u^{4} + \ldots
$$
As a result, formula (\ref{eq:Iharasplit}) implies 
$$
\zeta_{Y}(u) = \zeta_{X}(u)L(\pi,u) = \frac{1}{1 - 4u + 4u^{2} - 4u^{3} - 2u^{4} + 20u^{5} - 12u^{6} - 12u^{7} + 9u^{8}},
$$
confirming the computation above without having to produce the edge adjacency matrix for $Y$. 
\end{ex}

\begin{ex}\label{ex:complicated-voltage-cover}
For a more complicated example, consider the graph $X$ pictured below. 
\begin{center}
\begin{tikzpicture}[scale=1.5,decoration={markings,mark=at position 0.5 with {\arrow[scale=2]{>}}}]
    \fill (-.7,0) circle (.05);
    \fill (.7,0) circle (.05);
    \fill (0,1) circle (.05);
    \draw[postaction=decorate] (-.7,0) to[out=-50,in=230] (.7,0);
        \node at (0,-.5) {$a$};
    \draw[postaction=decorate] (-.7,0) -- (.7,0) node[below,pos=.55] {$b$};
    \draw[postaction=decorate] (-.7,0) to[out=50,in=130] (.7,0);
        \node at (0,.5) {$c$};
    \draw[postaction=decorate] (-.7,0) to[out=80,in=210] (0,1);
        \node at (-.7,.5) {$d$};
    \draw[postaction=decorate] (.7,0) to[out=100,in=330] (0,1);
        \node at (.7,.5) {$e$};
    \draw[postaction=decorate] (0,1) arc(-90:270:.3) node[above,pos=.5] {$f$};
    \node at (-1.5,.3) {$X$};
\end{tikzpicture}
\end{center}
It admits many double covers, such as: 
\begin{center}
\begin{tikzpicture}[scale=1.5,decoration={markings,mark=at position 0.5 with {\arrow[scale=2]{>}}}]
    \fill (-2.5,.7) circle (.05);
    \fill (-2.5,-.7) circle (.05);
    \fill (-.7,0) circle (.05);
    \fill (.7,0) circle (.05);
    \fill (2.5,.7) circle (.05);
    \fill (2.5,-.7) circle (.05);
    \draw[postaction=decorate] (-2.5,.7) to[out=240,in=120] (-2.5,-.7);
        \node at (-2.9,0) {$a^{+}$};
    \draw[postaction=decorate] (2.5,.7) to[out=300,in=60] (2.5,-.7);
        \node at (2.9,0) {$a^{-}$};
    \draw[postaction=decorate] (-2.5,.7) to[out=45,in=120] (2.5,-.7);
        \node at (.6,1.2) {$b^{+}$};
    \draw[postaction=decorate] (2.5,.7) to[out=135,in=60] (-2.5,-.7);
        \node at (-.6,1.2) {$b^{-}$};
    \draw[postaction=decorate] (-2.5,.7) to[out=300,in=60] (-2.5,-.7);
        \node at (-2.1,.15) {$c^{+}$};
    \draw[postaction=decorate] (2.5,.7) to[out=240,in=120] (2.5,-.7);
        \node at (2.1,.15) {$c^{-}$};
    \draw[postaction=decorate] (-2.5,.7) to[out=0,in=150] (.7,0);
        \node at (-.6,.7) {$d^{+}$};
    \draw[postaction=decorate] (2.5,.7) to[out=180,in=30] (-.7,0);
        \node at (.6,.7) {$d^{-}$};
    \draw[postaction=decorate] (-2.5,-.7) to[out=0,in=210] (.7,0);
        \node at (-.6,-.7) {$e^{+}$};
    \draw[postaction=decorate] (2.5,-.7) to[out=180,in=330] (-.7,0);
        \node at (.6,-.7) {$e^{-}$};
    \draw[postaction=decorate] (-.7,0) arc(0:360:.3) node[left,pos=.5] {$f^{+}$};
    \draw[postaction=decorate] (.7,0) arc(180:540:.3) node[right,pos=.5] {$f^{-}$};
    \node at (-3.5,0) {$Y$};
\end{tikzpicture}
\end{center}
The map $\pi : Y\rightarrow X$ sends each edge pair $x^{+},x^{-}$ to $x$, for $x\in\{a,b,c,d,e,f\}$; the map on vertices can be deduced from this. The edge adjacency matrix for $X$ is 
$$
T_{X} = \begin{pmatrix}
    0 & 0 & 0 & 1 & 0 & 1 & 0 & 1 & 0 & 0 & 0 & 0\\ 
    0 & 0 & 1 & 0 & 1 & 0 & 0 & 0 & 0 & 1 & 0 & 0\\ 
    0 & 1 & 0 & 0 & 0 & 1 & 0 & 1 & 0 & 0 & 0 & 0\\ 
    1 & 0 & 0 & 0 & 1 & 0 & 0 & 0 & 0 & 1 & 0 & 0\\ 
    0 & 1 & 0 & 1 & 0 & 0 & 0 & 1 & 0 & 0 & 0 & 0\\ 
    1 & 0 & 1 & 0 & 0 & 0 & 0 & 0 & 0 & 1 & 0 & 0\\ 
    0 & 1 & 0 & 1 & 0 & 1 & 0 & 0 & 0 & 0 & 0 & 0\\ 
    0 & 0 & 0 & 0 & 0 & 0 & 0 & 0 & 1 & 0 & 1 & 1\\ 
    1 & 0 & 1 & 0 & 1 & 0 & 0 & 0 & 0 & 0 & 0 & 0\\ 
    0 & 0 & 0 & 0 & 0 & 0 & 1 & 0 & 0 & 0 & 1 & 1\\ 
    0 & 0 & 0 & 0 & 0 & 0 & 1 & 0 & 1 & 0 & 1 & 0\\ 
    0 & 0 & 0 & 0 & 0 & 0 & 1 & 0 & 1 & 0 & 0 & 1   
\end{pmatrix}
$$
and so the Ihara zeta function for $X$ is 
\begin{align*}
    \zeta_{X}(u) &= \frac{1}{1 - 2u - 5u^{2} + 6u^{3} + 3u^{4} - 24u^{5} + 38u^{6} + 56u^{7} - 97u^{8} - 54u^{9} + 87u^{10} + 18u^{11} - 27u^{12}}\\
        &= 1 + 2u + 9u^{2} + 22u^{3} + 74u^{4} + \ldots
\end{align*}
One can obtain $Y$ as a \emph{voltage cover} of $X$; i.e., by choosing a voltage assignment for the edges of $X$ as below and allowing $Y$ to be the \emph{derived graph}. See e.g. \cite{gro} for a description. Here we assign voltages as in the following table:
$$
\begin{array}{c|c c c c c c}
\text{edge}         & a & b & c & d & e & f \\
\hline
\alpha(\text{edge}) & 1 & -1 & 1 & -1 & -1 & 1
\end{array}
$$
This determines the voltages of all of the directed edges of $X$, since the voltage of any edge is the inverse of the voltage of its opposite. This allows one to compute the Artin--Ihara matrix for $\pi$ by multiplying through each row of $T_X$ by the voltage assigned to the corresponding edge. In this case, this means that the rows corresponding to $b$, $d$, and $e$, as well as the rows corresponding to their opposite edges $b^\iota$, $d^\iota$, and $e^\iota$, get multiplied by $-1$, while the rest remain the same. Explicitly,
$$
T(\pi) = \begin{pmatrix}
    0 & 0 & 0 & 1 & 0 & 1 & 0 & 1 & 0 & 0 & 0 & 0\\ 
    0 & 0 & 1 & 0 & 1 & 0 & 0 & 0 & 0 & 1 & 0 & 0\\ 
    0 & -1 & 0 & 0 & 0 & -1 & 0 & -1 & 0 & 0 & 0 & 0\\ 
    -1 & 0 & 0 & 0 & -1 & 0 & 0 & 0 & 0 & -1 & 0 & 0\\ 
    0 & 1 & 0 & 1 & 0 & 0 & 0 & 1 & 0 & 0 & 0 & 0\\ 
    1 & 0 & 1 & 0 & 0 & 0 & 0 & 0 & 0 & 1 & 0 & 0\\ 
    0 & -1 & 0 & -1 & 0 & -1 & 0 & 0 & 0 & 0 & 0 & 0\\ 
    0 & 0 & 0 & 0 & 0 & 0 & 0 & 0 & -1 & 0 & -1 & -1\\ 
    -1 & 0 & -1 & 0 & -1 & 0 & 0 & 0 & 0 & 0 & 0 & 0\\ 
    0 & 0 & 0 & 0 & 0 & 0 & -1 & 0 & 0 & 0 & -1 & -1\\ 
    0 & 0 & 0 & 0 & 0 & 0 & 1 & 0 & 1 & 0 & 1 & 0\\ 
    0 & 0 & 0 & 0 & 0 & 0 & 1 & 0 & 1 & 0 & 0 & 1   
\end{pmatrix}
$$
and its reciprocal characteristic polynomial gives the Artin--Ihara $L$-function for $\pi$: 
\begin{align*}
    L(\pi,u) &= \frac{1}{1 - 2u + 3u^{2} - 6u^{3} + 3u^{4} + 12u^{5} - 10u^{6} + 20u^{7} - 33u^{8} - 42u^{9} + 63u^{10} + 18u^{11} - 27u^{12}}\\
        &= 1 + 2u + u^{2} + 2u^{3} + 10u^{4} + \ldots
\end{align*}
One can also deduce this from the formula $f(n) = \#S_{1}(n)^{+} - \#S_{1}(n)^{-}$, using the fact that an effective $0$-cycle lies in $S_{1}(n)^{+}$ (resp.~$S_{1}(n)^{-}$) if it contains an even (resp.~odd) number of charge-reversing edges, which in this example are the edges $b,b^{\iota},d,d^{\iota},e,e^{\iota}$. By formula (\ref{eq:Iharasplit}), we can compute the Ihara zeta function of $Y$ without computing its edge adjacency matrix: 
$$
\zeta_{Y}(u) = \zeta_{X}(u)L(\pi,u) = \frac{1}{1 - 4u + \ldots + 729u^{24}} = 1 + 4u + 14u^{2} + 44u^{3} + 141u^{4} + \ldots
$$
\end{ex}

As the number of edges in $X$ increases, this becomes more efficient than computing $\zeta_{Y}(u)$ directly. For more examples of computations with Artin--Ihara $L$-functions, see \cite[Sec.~18.4]{ter}.

\subsection{Supersingular Isogeny Graphs}
\label{sec:ss}

Fix distinct primes $p$ and $\ell$ and consider the graph $G = G_{p}(\ell)$ defined as follows. The vertices of $G$ are the supersingular elliptic curves over $\F_{p^{2}}$ (which can be labeled by their $j$-invariants) and two vertices are connected by an edge for each $\ell$-isogeny between the curves. For large $p > \ell$, these graphs appear in the context of cryptography, for example in \cite{clg}. 

In our present context, a key feature of $G$ was identified by Ribet in \cite{rib}, where he showed that $G$ admits a double cover $\widetilde{G}\rightarrow G$ by the dual graph $\widetilde{G}$ of the special fibre at $\ell$ of a Shimura curve $X_{0}^{p\ell}(1)$. As a result, we can use formula (\ref{eq:Iharasplit}) to describe the zeta function of $\widetilde{G}$ in terms of $G$, i.e.~without reference to the Shimura curve itself. 

\begin{ex}
\label{ex:G213}
For $\ell = 2$ and $p = 13$, the supersingular isogeny graph and the dual graph of the special fibre of the Shimura curve are not too hard to draw: 
\begin{center}
\begin{tikzpicture}
    \fill (0,0) circle (.07);
    \fill (2,0) circle (.07);
    \draw (0,0) -- (2,0) node[below,pos=.5] {$a$};
    \draw (0,0) to[out=45,in=135] (2,0);
        \node at (1,.7) {$b^{+}$};
    \draw (0,0) to[out=-45,in=-135] (2,0);
        \node at (1,-.7) {$b^{-}$};
    \node at (-1,0) {$\widetilde{G}$};
    \draw[->] (3,0) -- (4,0);
\end{tikzpicture}
\hspace{.3in}
\begin{tikzpicture}
    \fill (0,0) circle (.07);
    \draw (0,0) arc(180:540:.5) node[right,pos=.5] {$a$};
    \draw (0,0) arc(180:540:1) node[right,pos=.5] {$b$};
    \node at (3,0) {$G$};
\end{tikzpicture}
\end{center}
The Ihara zeta functions and Artin--Ihara $L$-function for the cover $\pi : \widetilde{G}\rightarrow G$ are also easy to compute (although see Remark~\ref{rem:oppositeedge}): 
\begin{align*}
    \zeta_{G}(u) &= \frac{1}{1 - 2u - u^{2} + 2u^{3}} = 1 + 2u + 5u^{2} + 10u^{3} + 21u^{4} + \ldots\\
    L(\pi,u) &= \frac{1}{1 + 2u - u^{2} - 2u^{3}} = 1 - 2u + 5u^{2} - 10u^{3} + 21u^{4} - \ldots\\
    \zeta_{\widetilde{G}}(u) &= \frac{1}{1 - 6u^{2} + 9u^{4} - 4u^{6}} = 1 + 6u^{2} + 27u^{4} + 112u^{6} + \ldots
\end{align*}
Of course, we could have computed $\zeta_{\widetilde{G}}(u)$ indirectly, using formula (\ref{eq:Iharasplit}), as shown in the next example. 
\end{ex}

\begin{ex}
\label{ex:G261}
For $\ell = 2$ and $p = 61$, the double cover looks like
\begin{center}
\begin{tikzpicture}
    \fill (0,3) circle (.07);
    \fill (1,3.7) circle (.07);
    \fill (1,2.3) circle (.07);
    \fill (2,3) circle (.07);
    \fill (3.5,3) circle (.07);
    \fill (0,5) circle (.07);
    \fill (1,5.7) circle (.07);
    \fill (1,4.3) circle (.07);
    \fill (2,5) circle (.07);
    \fill (3.5,5) circle (.07);
    \draw (0,5) -- (0,3) node[left,pos=.5] {$a$};
    \draw (0,5) -- (1,5.7) node[above,pos=.4] {$b^{+}$};
    \draw (0,3) -- (1,2.3) node[below,pos=.4] {$b^{-}$};
    \draw (0,5) -- (1,4.3) node[below,pos=.3] {$c^{+}$};
    \draw (0,3) -- (1,3.7) node[above,pos=.3] {$c^{-}$};
    \draw (1,5.7) to[out=-60,in=60] (1,3.7);
        \node at (1,5) {$d^{+}$};
    \draw (1,4.3) to[out=-120,in=120] (1,2.3);
        \node at (1.1,3) {$d^{-}$};
    \draw (1,5.7) -- (2,5) node[above,pos=.6] {$e^{+}$};
    \draw (1,2.3) -- (2,3) node[below,pos=.6] {$e^{-}$};
    \draw (1,4.3) -- (2,5) node[below,pos=.7] {$f^{+}$};
    \draw (1,3.7) -- (2,3) node[above,pos=.7] {$f^{-}$};
    \draw (2,5) -- (3.5,5) node[above,pos=.5] {$g^{+}$};
    \draw (2,3) -- (3.5,3) node[below,pos=.5] {$g^{-}$};
    \draw (3.5,5) to[out=-60,in=60] (3.5,3);
        \node at (4.1,4) {$h^{+}$};
    \draw (3.5,5) to[out=-120,in=120] (3.5,3);
        \node at (2.9,4) {$h^{-}$};
    \node at (5.5,4) (GG) {$\widetilde{G}$};
    \fill (0,0) circle (.07);
    \fill (1,.7) circle (.07);
    \fill (1,-.7) circle (.07);
    \fill (2,0) circle (.07);
    \fill (3.5,0) circle (.07);
    \draw (0,0) arc (0:360:.5) node[left,pos=.5] {$a$};
    \draw (0,0) -- (1,.7) node[above,pos=.5] {$b$};
    \draw (0,0) -- (1,-.7) node[below,pos=.5] {$c$};
    \draw (1,.7) -- (1,-.7) node[right,pos=.5] {$d$};
    \draw (1,.7) -- (2,0) node[above,pos=.5] {$e$};
    \draw (1,-.7) -- (2,0) node[below,pos=.5] {$f$};
    \draw (2,0) -- (3.5,0) node[above,pos=.5] {$g$};
    \draw (3.5,0) arc (180:540:.5) node[right,pos=.5] {$h$};
    \node at (5.5,0) (G) {$G$};
    \draw[->] (GG) -- (G);
\end{tikzpicture}
\end{center}
Since we know the dual graph of the special fibre of $X(122)$, we could compute its Ihara zeta function directly, but we will instead use formula (\ref{eq:Iharasplit}) to simplify the computation. Using Magma, we compute 
$$
\zeta_{G}(u) = \frac{1}{1 - 2u + \ldots - 32u^{15}} = 1 + 2u + 3u^{2} + 8u^{3} + 17u^{4} + \ldots
$$
This says that, for example, there are $2$ primes of length $1$, no primes of length $2$ (the $3$ effective $0$-cycles are all composites of the length $1$ primes), etc. Treating $\pi : \widetilde{G}\rightarrow G$ as a voltage cover as in Example \ref{ex:complicated-voltage-cover}, it is straightforward to compute 
$$
L(\pi,u) = \frac{1}{1 + 2u + \ldots - 32u^{15}} = 1 - 2u + 3u^{2} - 8u^{3} + 17u^{4} + \ldots
$$
Then by formula (\ref{eq:Iharasplit}), 
$$
\zeta_{\widetilde{G}}(u) = \zeta_{G}(u)L(\pi,u) = 1 + 2u^{2} + 11u^{4} + 44u^{6} + \ldots
$$
This can be confirmed by computing the Ihara zeta function of $\widetilde{G}$ directly, though it is not recommended. 
\end{ex}

\begin{rem}
\label{rem:oppositeedge}
There is a subtle technical point in the above computations. When forming the edge adjacency matrix for $G$, we are treating the loop $a$ as its own opposite - the corresponding isogeny is its own dual - while the loops $b$ in Example~\ref{ex:G213} and $h$ in Example~\ref{ex:G261} each have a distinct opposite, corresponding to a dual isogeny, which is not shown. This isn't explored directly in \cite{ter}, but it is in \cite{zak}. 
\end{rem}

It follows from Ribet's work \cite{rib} that $\widetilde{G}$ is the unique connected, bipartite cover of $G$, up to isomorphism. In the language of voltage covers, this means $\widetilde{G}$ is always isomorphic to the voltage cover associated to the voltage which assigns a $-1$ charge to each edge of $G$.
\footnote{Note that two different voltage assignments may give isomorphic covers. For example, the figure in Example~\ref{ex:G261} implies a voltage assignment in which only the edges $a$, $d$, and $h$ (and their opposites) are assigned a voltage of $-1$, while the rest are assigned a voltage of $+1$. However, one may check that the cover is bipartite, and thus is isomorphic to the cover produced by a voltage assignment where every edge is given a voltage of $-1 \in (\{\pm1\},\cdot)$.}
In the language of Subsection~\ref{sec:graphs}, all primes of $G$ of odd length are inert in the cover $\widetilde{G}\rightarrow G$, while all primes of even length are split.

Translated to zeta functions, this implies the following. 

\begin{prop}
\label{prop:sszeta}
For distinct primes $p$ and $\ell$, let $G = G_{p}(\ell)$ be the mod $p$ supersingular $\ell$-isogeny graph and let $\widetilde{G}$ be the dual graph of the special fibre of the Shimura curve $X_{0}^{p\ell}(1)$ at $\ell$. Then 
$$
\zeta_{\widetilde{G}}(u) = \zeta_{G}(u)\zeta_{G}(-u). 
$$
\end{prop}

In particular, in $\zeta_{\widetilde{G}}(u)$, the coefficients of odd powers of $u$ are all $0$, reflecting the fact that $\widetilde{G}$ is bipartite. We can also extract the following recursions on coefficients in $\zeta_{\widetilde{G}}(u)$. 

\begin{cor}
Write $\zeta_{G}(u) = 1 + \sum_{n = 1}^{\infty} a_{n}u^{n}$ and $\zeta_{\widetilde{G}}(u) = 1 + \sum_{n = 1}^{\infty} b_{n}u^{n}$. Then for all even $n > 0$, 
$$
b_{n} = \sum_{i = 0}^{n/2 - 1} (-1)^{i}2a_{i}a_{n/2 - i} + (-1)^{n/2}a_{n/2}^{2}. 
$$
\end{cor}

In arithmetic terms, the coefficient $a_{n}$ counts the number of effective $0$-cycles in $G$ of length $n$. On the other hand, $\zeta_{G}(u)$ also encodes the count of {\it $\ell$-isogeny cycles} in $G$, as we explain now. Writing $\zeta_{G}(u)$ in exponential notation as 
$$
\zeta_{G}(u) = \exp\left [\sum_{m = 1}^{\infty} \frac{N_{G}(m)}{m}u^{m}\right ],
$$
the coefficients $N_{G}(m)$ count the number of closed walks of length $m$ in $G$ without backtracking or tails \cite[Def.~4.2]{ter}. 

\begin{cor}
\label{cor:NGformula}
For all $m\geq 1$, $N_{\widetilde{G}}(m) = (1 + (-1)^{m})N_{G}(m)$. 
\end{cor}

Writing $\pi_{G}(m)$ for the number of primes of $G$ of length $m$, we have 
\begin{equation}\label{eq:Nformula}
    N_{G}(m) = \sum_{d\mid m} d\pi_{G}(d),
\end{equation}
so by M\"{o}bius inversion, 
\begin{equation}\label{eq:piformula}
    \pi_{G}(m) = \frac{1}{m}\sum_{d\mid m} \mu(m/d)N_{G}(d)
\end{equation}
where $\mu$ is the classical M\"{o}bius function [{\it loc.~cit.}, Ch.~10]. 

When $G$ is a supersingular isogeny graph, $\pi_{G}(m)$ equivalently counts the number of $\ell$-isogeny cycles of length $m$ in $G$, in the sense of \cite[Def.~3.1]{aclsst}. Asymptotics for the function $\pi_{G}(m)$ are known. For example, the authors in \cite{aclsst} show the following. 

\begin{thm}[{\cite[Thm.~7.1]{aclsst}}]
\label{thm:aclsst}
For primes $p\equiv 1\pmod{12}$, as $m\rightarrow\infty$, $\pi_{G}(m)$ is asymptotically $\ell^{m}/2m$. 
\end{thm}

This result is deduced in [{\it loc.~cit.}] using probabilistic methods. We give an independent proof here. 

\begin{proof}
When $p\equiv 1\pmod{12}$, $G = G_{p}(\ell)$ is an $(\ell + 1)$-regular graph, so the radius of convergence of $\zeta_{G}(u)$ is $\ell^{-1}$. Applying the {\it prime number theorem for graphs} \cite[Thm.~10.1]{ter} to $\zeta_{G}(u)$ gives the asymptotic. 
\end{proof}

As a consequence of this and Corollary~\ref{cor:NGformula}, we obtain an asymptotic for primes in the graph $\widetilde{G}$. 

\begin{cor}
\label{cor:shimuraasymptotic}
Assume $p\equiv 1\pmod{12}$. As $m\rightarrow\infty$ for even $m$, $\pi_{\widetilde{G}}(m)$ is asymptotically $\ell^{m}/m$. 
\end{cor}

\begin{rem}\label{rem:irregular}
When $p\not\equiv 1\pmod{12}$, $G$ is irregular due to the presence of extra automorphisms for supersingular elliptic curves with $j$-invariants $0$ and $1728$. 
As a result, the two proofs of Theorem~\ref{thm:aclsst} fail for these $p$. On the other hand, it is known \cite[Lem.~22.1]{ter} that if $\widetilde{G}\rightarrow G$ is a cover of graphs (of any degree), then $\zeta_{G}(u)$ and $\zeta_{\widetilde{G}}(u)$ have the same radius of convergence. In our situation, this would seem to imply that the radius of convergence of $\zeta_{G}(u)$ equals $\ell^{-1}$ since the covering graph $\widetilde{G}$ is always $(\ell + 1)$-regular. As a consequence, we would be able to deduce the asymptotic \cite[Thm.~7.1]{aclsst} for all primes $p$. 

However, there is a tension here, as a Galois double cover of a graph $G$ is regular if and only if $G$ itself is regular. This appears to be an oversight in the way Ribet's correspondence is used in this context: if automorphisms are taken into account, there cannot be a $2$-to-$1$ correspondence on edges. We plan to resolve this issue by including the extra automorphisms directly in the structure of $G$, which will appear in a forthcoming sequel to the present paper along with a version of Theorem~\ref{thm:aclsst} for all primes $p$. We expect Corollary~\ref{cor:shimuraasymptotic} to hold for all $p$ as well. 

As further evidence, the authors in \cite{aclsst} predict that the asymptotic in Theorem~\ref{thm:aclsst} holds for all $p \gg \ell$ since the phenomenon of extra automorphisms at $j = 0$ and $1728$ becomes sparse as $p$ grows. Empirical evidence suggests this should hold even for small $p$ when automorphisms are accounted for. 
\end{rem}

%
%

\begin{rem}
\label{rem:othersscovers}
Our work above with Ribet's double cover $\widetilde{G}\rightarrow G$ already makes it clear that formula (\ref{eq:Iharasplit}) for double covers of graphs has useful applications to supersingular isogeny graphs. In fact, there are several other ways in which covers of graphs make an appearance in the theory: 
\begin{enumerate}[(1)]
    \item The supersingular $\ell$-isogeny graph is almost a double cover of the quaternion algebra $\ell$-ideal graph $G' = G_{p}'(\ell)$, as in \cite{ailms}. More precisely, $G\rightarrow G'$ fails to be $2$-to-$1$ at the $\F_{p}$-points of $G$. As such, this is a natural candidate for some theory of ramified covers of graphs (see Section~\ref{sec:future}). 
    \item One can also add level structure to the graph $G = G_{p}(\ell)$, as in \cite{arp}. For the graph $G_{p}(\ell,N)$ of $\ell$-isogenies between supersingular elliptic curves with level $\Gamma_{0}(N)$ structure, where $N$ is prime to $p$ and $\ell$, there is a graph morphism $G_{p}(\ell,N)\rightarrow G_{p}(\ell)$ which is a degree $N + 1$ Galois cover away from $j = 0,1728$. Therefore, when $p\equiv 1\pmod{12}$, a version of formula (\ref{eq:Iharasplit}) for higher degree covers would imply asymptotics of isogeny cycles in $G_{p}(\ell,N)$. For other $p$, such asymptotics would follow from a version of (\ref{eq:Iharasplit}) for ramified covers. 
    \item Generalizing the $\ell$-isogeny context slightly, one can form a graph $G_{p}(N)$ of cyclic $N$-isogenies among the supersingular elliptic curves over $\overline{\F}_{p}$, for arbitrary $N$ (prime to $p$, say). Then it is natural to ask if there is a Galois cover of graphs $G_{p}(N')\rightarrow G_{p}(N)$ whenever $N\mid N'$. Passing to double covers, viewed as dual graphs of specializations of Shimura curves, one would expect such a cover whenever there is a (possibly ramified) covering map between the Shimura curves. 
    \item In unpublished notes \cite{cla}, Clark also gives an interpretation of supersingular isogeny graphs in terms of \emph{good} reduction of Shimura curves. In this interpretation, a Galois cover of supersingular isogeny graphs would come from an \'{e}tale cover of Shimura curves, and divisibility of their zeta functions is automatic. 
\end{enumerate}
\end{rem}

\begin{rem}
\label{rem:modular}
Another connection between Ihara zeta functions of supersingular isogeny graphs and Hasse--Weil zeta functions of modular curves is established in \cite{sug} and extended to supersingular isogeny graphs with level structure and Shimura curves in \cite{lei-muller}. Their main formula \cite[Cor.~3.4]{lei-muller} is an algebraic relation between zeta functions and as such should have an objective analogue of a similar form as our formula (\ref{eq:genquad}). We will investigate this relationship, as well as the applications to modular forms obtained in \cite{sug}, in future work. 
\end{rem}

\subsection{Quadratic Reciprocity}
\label{sec:recip}

The famous quadratic reciprocity law of Gauss states that for distinct, odd primes $p$ and $q$, their Legendre symbols are intertwined according to 
$$
\legen{p}{q}\legen{q}{p} = (-1)^{(p - 1)(q - 1)/4}. 
$$
Setting $p^{*} = (-1)^{(p - 1)/2}p$ gives us a more compact version of the formula: 
$$
\legen{p^{*}}{q} = \legen{q}{p}. 
$$

Let $K = \Q\left (\sqrt{p^{*}}\right )$. By Theorem~\ref{thm:mainthmloc} or formula (\ref{eq:quadzetalocal}), $\zeta_{K,q}(s) = \zeta_{\Q,q}(s)L_{q}(\chi,s)$ for any prime $q$, where $\chi = \chi^{+} - \chi^{-}$. Then 
$$
\chi(q) = \begin{cases}
    0, &\text{if $q$ is ramified, i.e.~$q = p$}\\
    1, &\text{if $q$ is split}\\
    -1, &\text{if $q$ is inert}
\end{cases}
$$
and by the usual description of splitting types in quadratic extensions, we see that $\chi(q) = \legen{p^{*}}{q}$ for all $q$. On the other hand, the map $\psi : \Z\rightarrow\C^{\times},n\mapsto\legen{n}{p}$ is the unique quadratic Dirichlet character with conductor $p$, so its Dirichlet $L$-function 
$$
L(\psi,s) = \prod_{q} (1 - \psi(q)q^{-s})^{-1}
$$
also satisfies $\zeta_{\Q}(s)L(\psi,s) = \zeta_{K}(s)$; see Remark~\ref{rem:charsplit} below. Matching up $q^{s}$ coefficients produces 
$$
\legen{p^{*}}{q} = \chi(q) = \psi(q) = \legen{q}{p}. 
$$
Thus the algebraic relation between the zeta functions $\zeta_{K}(s)$ and $\zeta_{\Q}(s)$ allow one to ``rediscover'' quadratic reciprocity. 

\begin{rem}
\label{rem:charsplit}
The usual way to prove $\psi(q) = \chi(q)$ is to show that both are equal to $\Frob_{q}(K/\Q)$; $\psi(q) = \Frob_{q}$ follows from the restriction of Frobenius from the cyclotomic extension $\Q(\zeta_{p})/\Q$ to its unique quadratic subfield, which is $K$. However, we can prove it directly for all unramified $q$ using only that $\psi$ is the unique nontrivial primitive quadratic character of conductor $p$. We will show that for unramified $q$, 
\begin{equation}\label{eq:charsplit}
\zeta_{K,q}(s) = \zeta_{\Q,q}(s)L_{q}(\psi,s). 
\end{equation}
(Of course, this formula follows immediately from the fact that the permutation representation of $\Gal(K/\Q)$ splits as ${\bf 1}\oplus\psi$ where ${\bf 1}$ is the trivial representation, but let's do things from scratch.) If $X(K) = \{{\bf 1},\psi\}$ is the group of characters of $\Gal(K/\Q)$, then the map $X(K) \rightarrow \{\pm 1\}\subset\C^{\times}$ defined by $\rho\mapsto\rho(\Frob_{p})$ is a surjective homomorphism of groups which is well-defined whenever $p$ is unramified in $K$. Let $f_{p}$ be the inertia degree at $p$ and $r_{p}$ the number of primes of $\orb_{K}$ lying over $p$, so that $f_{p}r_{p} = 2$. Then we can expand the right side of formula (\ref{eq:charsplit}): 
\begin{align*}
    \zeta_{\Q,p}(s)L_{p}(\psi,s) &= (1 - p^{-s})^{-1}(1 - \psi(p)p^{-s})^{-1} = \prod_{\alpha^{f_{p}} = 1} (1 - \alpha p^{-s})^{-r_{p}}\\
        &= (1 - p^{-f_{p}s})^{-r_{p}} = \prod_{\frak{p}\mid p} (1 - N(\frak{p})^{-s})^{-1}. 
\end{align*}
By Proposition~\ref{prop:quadsplitting}, this is equal to $\zeta_{K,p}(s)$, so formula (\ref{eq:charsplit}) is established. 
\end{rem}

\begin{rem}
A similar argument using $K = \Q(i)$ (resp.~$K = \Q\left (\sqrt{2}\right )$) gives the supplementary laws for $\legen{-1}{p}$ (resp.~$\legen{2}{p}$). 
\end{rem}

A key ingredient here is that $K = \Q\left (\sqrt{p^{*}}\right )$ is the unique quadratic extension of $\Q$ in which $p$ is the only ramified (finite) prime. The argument above also works for any quadratic extension $K/\Q$ in which $p$ ramifies and $q$ does not. 

When $K/\Q$ is replaced by a hyperelliptic curve $C\rightarrow\P^{1}$ over $\F_{q}$, we recover the analogue of quadratic reciprocity for the function field extension $\F_{q}(C)/\F_{q}(t)$. For the rest of this discussion, assume $\char\F_{q} = p\not = 2$. 

Explicitly, fix distinct points $x,y\in\P^{1}(\overline{\F}_{q})$ and set $\deg(x) = d$ and $\deg(y) = e$; for simplicity, we may assume $x,y\not = \infty = [0 : 1]$, i.e.~$x$ and $y$ lie in the affine open $U = \Spec\F_{q}[t] = \A^{1}\subseteq\P^{1}$. Let $C$ be a hyperelliptic curve which is ramified at $x$ but not $y$, e.g.~the Legendre elliptic curve\footnote{Recall that we are allowing elliptic curves to be considered hyperelliptic.} defined by $s^{2} = t(t - 1)(t - x)$ if $x,y\not = [1 : 0]$ or $[1 : 1]$ (which can be achieved after a change of coordinates). Let $\pi : C\rightarrow\P^{1}$ be the usual hyperelliptic map induced by $(t,s)\mapsto [t : s : 1]$. By formula (\ref{eq:quadzetalocal}), we have $Z_{y}(C,T) = Z_{y}(\P^{1},T)L_{y}(C,T)$, where 
$$
L_{y}(C,T) = (1 - \chi(F_{y})T^{\deg(y)})^{-1} \quad\text{for the character}\quad \chi(F_{y}) = \begin{cases}
    0, &\text{if $y$ is ramified}\\
    1, &\text{if $y$ is split}\\
    -1, &\text{if $y$ is inert}.
\end{cases}
$$

Let $F = \F_{q}(t)$ and $K = \F_{q}(C)$ be the function fields of $\P^{1}$ and $C$. On $\P^{1}$, $x$ corresponds to an irreducible polynomial $g(t)\in\F_{q}[t]$ of degree $d$. Consider the quadratic character $\psi = \legen{\cdot}{g} : \F_{q}[t]\rightarrow\C^{\times}$ defined by 
$$
\psi(h) = \legen{h}{g} = \begin{cases}
    0, &\text{if } h = g\\
    1, &\text{if $h$ is a square in } \F_{q}[t]/g\F_{q}[t]\\
    -1, &\text{if $h$ is not a square in } \F_{q}[t]/g\F_{q}[t]. 
\end{cases}
$$
Alternatively, this can be defined by $\legen{h}{g} \equiv h^{(q^{d} - 1)/2} \pmod{g\F_{q}[t]}$. One can show that $\psi$ is a Dirichlet character mod $g$ on the group $(\F_{q}[t]/g\F_{q}[t])^{\times}$. Along the same lines as Remark~\ref{rem:charsplit}, its $L$-function 
$$
L(\psi,T) = \prod_{h\in\F_{q}[t]} (1 - \psi(h)T^{\deg(h)})^{-1}
$$
satisfies $Z_{y}(C,T) = Z_{y}(\P^{1},T)L_{y}(\psi,T)$ for \emph{any} $y\in U$, where $L_{y}(\psi,T) = (1 - \psi(h)T^{\deg(h)})^{-1}$ if $h$ is the irreducible polynomial in $\F_{q}[t]$ corresponding to $y\in U$. Comparing this to formula (\ref{eq:quadzetalocal}), we get $\psi(y) = \chi(F_{y})$. Further, for unramified $y\in U$ with corresponding irreducible polynomial $h$, 
\begin{align*}
    y \text{ is split } &\iff g^{*} \text{ is a square in } \F_{q}[t]/h\F_{q}[t]\\
    y \text{ is inert } &\iff g^{*} \text{ is not a square in } \F_{q}[t]/h\F_{q}[t]
\end{align*}
where $g^{*} = (-1)^{(q^{d} - 1)/2}g$. Therefore $\chi(F_{y}) = \legen{g^{*}}{h}$. Putting it all together, we see that 
$$
\legen{g^{*}}{h} = \chi(F_{y}) = \psi(h) = \legen{h}{g},
$$
which is the function field version of quadratic reciprocity \cite[Ch.~3]{ros}. 

\begin{rem}
Replacing $\F_{q}(C)$ with $K = \F_{q}(t)\left (\sqrt{u}\right )$ for a unit $u\in\F_{q}^{\times}$ in the argument above yields the supplementary law 
$$
\legen{u}{g} = u^{(q - 1)\deg(g)/2}
$$
for all monic irreducible $g\in\F_{q}[t]$. Geometrically, this corresponds to the base change $\P_{k}^{1}\rightarrow\P_{\F_{q}}^{1}$ where $k = \F_{q}\left (\sqrt{u}\right )$, which is an ``arithmetic double cover''. 
\end{rem}

For finite graphs, there are no notions of function fields and coordinate rings to work with. More disheartening is the lack of a clear notion of ramification, so that constructing a quadratic cover ramified at a fixed prime has no natural analogue. However, there is at least one candidate for a definition of ramified cover in the literature \cite{ura,bn}; see Section~\ref{sec:ramgraph} for further discussion.

\section{Future Directions}
\label{sec:future}

\subsection{Higher Degree Extensions}

As explained in Section~\ref{sec:intro}, often when there is a map $Y\rightarrow X$ with nice properties (such as a covering space) the zeta function of $Y$ factors into a product of $L$-functions over $X$. We can categorify the various zeta and $L$-functions in this paper essentially because their coefficients are cardinalities of subsets of $1$-simplices in the relevant simplicial set. Equivalently, the ``character'' $\chi$ measuring these coefficients is quadratic. In the degree $\geq 3$ situation, $L(\pi)^{+} - L(\pi)^{-}$ will need to be replaced with a more general categorical $L$-functor (see Appendix~\ref{app:OLA}). 

\subsection{$L$-Functions}

In modern number theory and arithmetic geometry, results for zeta functions are often subsumed by more general statements for $L$-functions, making these a natural target for our categorification program. Fortunately, we have made progress towards a theory of {\it $L$-functors} which adapts the techniques in the present work to a more general framework that accommodates $L$-functions. We will present this work in a future article, using an objective linear algebra theory that replaces simplicial sets with simplicial objects in other categories (see Appendix~\ref{app:OLA} for a brief introduction). For Artin $L$-functions, their corresponding $L$-functors are linear functors of admissible Galois representations. 

\subsection{Motivic Zeta and $L$-Functions}

To push beyond the applications in this article, it's natural to ask for a further categorification of motivic zeta and $L$-functions, as suggested in \cite[Sec.~4]{kob}. In this direction, the authors in \cite{dh} have constructed a numerical version of Kapranov's motivic zeta function $Z_{mot}(X,t)$ for a $k$-variety $X$ that lies in an incidence algebra\footnote{The authors in \cite{dh} call this the {\it reduced incidence algebra for the poscheme of effective $0$-cycles} of $X$.} constructed from the Grothendieck ring of $k$-varieties, $K_{0}(\cat{Var}_{k})$. As $K_{0}(\cat{Var}_{k})$ is itself a decategorification of the category $\cat{Var}_{k}$, it should still be possible to lift $Z_{mot}(X,t)$ further to an objective incidence algebra, as originally suggested in \cite{kob}. Using the language of $L$-functors to extend the description in Subsection~\ref{sec:motivic}, we will offer such a construction in future work. This will also provide a natural home for motivic $L$-functions, defined as $L(X,V,t) = Z_{mot}((X\otimes V)^{G},t)$ for an algebraic group $G$ acting on $X$ and a $G$-representation $V$. We also plan to unify these with the $L$-functions in the previous paragraph by extending the formalism to stacks. 

\subsection{Supersingular Isogeny Graphs}

We saw in Subsection~\ref{sec:ss} that formula (\ref{eq:Iharasplit}) has practical applications in the theory of supersingular isogeny graphs, specifically when applied to the double cover $\widetilde{G}\rightarrow G_{p}(\ell)$ of the mod $p$ supersingular $\ell$-isogeny graph by the dual fibre of the specialization of the Shimura curve $X_{0}^{p\ell}(1)$ at $\ell$. In Remark~\ref{rem:othersscovers}, we also identified other types of covers that might prove useful in this theory. In particular, the morphism of graphs $G_{p}(\ell)\rightarrow G_{p}'(\ell)$ from \cite{ailms} relates $\ell$-isogenies of supersingular elliptic curves to $\ell$-isogenies of ideals in quaternionic orders, while the morphism $G_{p}(\ell,N)\rightarrow G_{p}(\ell)$ introduces level $\Gamma_{0}(N)$ structure. We propose to study these morphisms as ramified covers of graphs (see below) and extract further asymptotic information for isogeny cycles, hopefully generalizing Theorem~\ref{thm:aclsst}. 

There are many further generalizations to consider. For principally polarized abelian varieties of dimension $g > 1$, one can form the isogeny graph on the \emph{superspecial} locus in $A_{g}$, as in \cite{jz1}, and get similar results for asymptotics of isogeny cycles as well as cycles in a corresponding double cover $\widetilde{G}\rightarrow G$. Dropping the principally polarized and superspecial requirements, one can also form a \emph{simplicial complex} of polarized supersingular isogenies \cite{jz2}, generalizing the simplicial structure of the graph in the $g = 1$ case. We plan to adapt our simplicial approach to study zeta functions of such simplicial complexes in a unified fashion. 

\subsection{Ramified Covers of Graphs}
\label{sec:ramgraph}

It would be natural to ask for a version of the theory outlined in Section~\ref{sec:graphs} for ramified covers of graphs. In particular, given a ramified version of formula (\ref{eq:Iharasplit}), one could obtain more information about loop counts, as in Section~\ref{sec:loopcounts}, and an analogue of quadratic reciprocity. One possible definition of a ramified cover of graphs is introduced in \cite{ura}, refined in \cite{bn} and used in \cite{mm} to prove an ``almost divisibility'' relation for Ihara zeta functions in a ramified cover of regular graphs, analogous to formula (\ref{eq:Iharasplit}). The examples below illustrate this phenomenon explicitly. 

\begin{ex}
Consider the bouquet of two loops $Y = B_{2}$ mapping to the cycle graph $X = B_{1}$: 
\begin{center}
\begin{tikzpicture}[decoration={markings,mark=at position 0.5 with {\arrow[scale=2]{>}}}]
    \fill (0,0) circle (.07);
    \draw[postaction=decorate] (0,0) arc(180:540:.5) node[right,pos=.5] {$a$};
    \draw[postaction=decorate] (0,0) arc(180:540:1) node[right,pos=.5] {$b$};
    \node at (-1,0) {$Y$};
    \draw[->] (3,0) -- (4,0) node[above,pos=.5] {$\pi$};
    \fill (5,0) circle (.07);
    \draw[postaction=decorate] (5,0) arc(180:540:.5) node[right,pos=.5] {$a$};
\end{tikzpicture}
\end{center}
This map is not a covering map in the sense of Section~\ref{sec:graphs}, but rather a {\it harmonic morphism}, which is one notion of ramified cover considered in \cite{ura,bn,mm}. It is easy to compute the Ihara zeta functions of $X$, 
$$
\zeta_{X}(u) = \frac{1}{(1 - u)^{2}} = 1 + 2u + 3u^{2} + 4u^{3} + \ldots
$$
and from Example~\ref{ex:loopcount}, we have that of $Y$, 
$$
\zeta_{Y}(u) = \frac{1}{1 - 4u + 2u^{2} + 4u^{3} - 3u^{4}} = 1 + 4u + 14u^{2} + 44u^{3} + \ldots
$$
In this case, one can check by hand that $\zeta_{X}(u)$ divides $\zeta_{Y}(u)$: 
$$
\frac{\zeta_{Y}(u)}{\zeta_{X}(u)} = \frac{1}{1 - 2u - 3u^{2}} = 1 + 2u + 7u^{2} + 20u^{3} + 61u^{4} + \ldots
$$
It is not clear how to interpret this as an $L$-function of the cover though. 
\end{ex}

\begin{ex}
For another example of a harmonic morphism, consider the bowtie $Y$ covering the cycle graph $X = C_{3}$: 
\begin{center}
\begin{tikzpicture}[decoration={markings,mark=at position 0.5 with {\arrow[scale=2]{>}}}]
    \fill (0,0) circle (.07);
    \fill (-1,.7) circle (.07);
    \fill (-1,-.7) circle (.07);
    \fill (1,.7) circle (.07);
    \fill (1,-.7) circle (.07);
    \draw[postaction=decorate] (0,0) -- (-1,.7);
    \draw[postaction=decorate] (-1,.7) -- (-1,-.7);
    \draw[postaction=decorate] (-1,-.7) -- (0,0);
    \draw[postaction=decorate] (0,0) -- (1,-.7);
    \draw[postaction=decorate] (1,-.7) -- (1,.7);
    \draw[postaction=decorate] (1,.7) -- (0,0);
    \draw[->] (2,0) -- (4,0) node[above,pos=.5] {$\pi$};
    \fill (5,0) circle (.07);
    \fill (6,.7) circle (.07);
    \fill (6,-.7) circle (.07);
    \draw[postaction=decorate] (5,0) -- (6,.7);
    \draw[postaction=decorate] (6,.7) -- (6,-.7);
    \draw[postaction=decorate] (6,-.7) -- (5,0);
\end{tikzpicture}
\end{center}
Here, we also have divisibility: 
\begin{align*}
    \zeta_{X}(u) &= \frac{1}{(1 - u^{3})^{2}} = 1 + 2u^{3} + 3u^{6} + 4u^{9} + \ldots\\
    \zeta_{Y}(u) &= \frac{1}{1 - 4u^{3} + 2u^{6} + 4u^{9} - 3u^{12}} = 1 + 4u^{3} + 14u^{6} + 44u^{9} + \ldots\\
    L(\pi,u) &= \frac{1}{1 - 2u^{3} - 3u^{6}} = 1 + 2u^{3} + 7u^{6} + 20u^{9} + \ldots
\end{align*}
Once again, it is not clear how to interpret the coefficients of $L(\pi,u)$ in terms of effective $0$-cycles on $X$. 
\end{ex}

\begin{ex}
(\cite[Fig.~1]{mm}) Let $X = K_{4}$ and let $Y$ be the ramified triple cover obtained by gluing $3$ copies of $X$ together at a single vertex. The zeta function of $X$ is 
$$
\zeta_{X}(u) = \frac{1}{1 - 8u^{3} - 6u^{4} - 16u^{6} + 24u^{7} - 3u^{8} - 16u^{9} - 24u^{10} + 16u^{12}} = 1 + 8u^{3} + 6u^{4} + 80u^{6} + \ldots
$$
while the zeta function of $Y$ is the reciprocal of a degree $36$ polynomial, which has been suppressed: 
$$
\zeta_{Y}(u) = 1 + 24u^{3} + 18u^{4} + \ldots
$$
However, the quotient $\zeta_{Y}(u)/\zeta_{X}(u)$ is not the reciprocal of a polynomial. Explicitly, 
$$
\frac{\zeta_{Y}(u)}{\zeta_{X}(u)} = \frac{1 - 2u}{P(u)} = 1 + 16u^{3} + 12u^{4} + 268u^{6} + \ldots
$$
where $P(u)$ is a degree $25$ polynomial in $u$. (This ``almost divisibility'' phenomenon is the main focus in \cite{mm}.) This means that whatever $L$-function we get in a hypothetical formula 
$$
\zeta_{Y}(u) = \zeta_{X}(u)L(\pi,u),
$$
it cannot be just a (product of) characteristic polynomial(s). 
\end{ex}

\begin{ex}
Something similar happens if you use a double cover $Y\rightarrow X = K_{4}$, in which case 
$$
\frac{\zeta_{Y}(u)}{\zeta_{X}(u)} = \frac{1 - 2x}{Q(u)} = 1 + 8u^{3} + 6u^{4} + 52u^{6} + \ldots
$$
where $Q(u)$ is a degree $13$ polynomial. 
\end{ex}

One possible solution to the failure of divisibility of Ihara zeta functions is to replace a ramified cover $Y\rightarrow X$ by a cover $Y\rightarrow\X$ where $\X$ is a {\it graph of groups}, as in \cite{zak}. Briefy, if $G$ is the automorphism group of the ramified cover $Y\rightarrow X$, there is a quotient object $\X = [Y/G]$ in the category of graphs with groups, keeping track of the orbits and stabilizers of the $G$-action. If $G$ acts without stabilizers on the edges of $Y$, \cite[Cor.~4.15]{zak} shows that $\zeta_{\X}(u)$ divides $\zeta_{Y}(u)$. 

We will investigate ramified covers, an objective analogue of \cite[Thm.~3.2]{mm} and \cite[Cor.~4.15]{zak}, and a version of quadratic reciprocity for prime cycles in future work, as well as the applications to supersingular isogeny graphs discussed in Remarks~\ref{rem:othersscovers} and~\ref{rem:modular}.

\appendix

\section{Objective Linear Algebra}
\label{app:OLA}

\subsection{Objective Linear Algebra Using the Category $\cat{Set}$}
\label{sec:OLASet}

In Section~\ref{sec:incalgs}, we constructed objective incidence algebras for the various decomposition sets attached to number fields, algebraic curves, CW complexes and graphs. The theory of {\it decomposition spaces} from \cite{gkt1,gkt2,gkt3,gkt-hla,gkt5} extends far beyond ``objective'' techniques in the category of sets, allowing for greater flexibility in our categorifications. To summarize the general theory, we reproduce below the dictionary from Section~\ref{sec:incalgs} in an arbitrary category of spaces.

\subsection{Replacing the Category of Sets}
\label{sec:OLAcat}

Let $(\mathcal{C},\otimes,\mathbb{1})$ be a symmetric monoidal category with coproduct $\oplus$. 
Furthermore, assume $\mathcal{C}$ has a terminal object $*$, and fix a section $s \colon * \to \mathbb{1}$. By abuse of notation, let $s$ also refer to the composition $A \to * \xrightarrow{s} \mathbb{1}$ for any object $A$ of $\mathcal{C}$.

With this setup, we revise the table given in section \ref{sec:OLASet}. 

\begin{center}
\begin{tabular}{|c|c|}
    \hline
    {\bf Linear} & {\bf Objective}\\
    \hline\hline
    field of scalars $k$ & a symmetric monoidal category $\mathcal{C}$\\
    scalar addition $+$ & $\oplus$\\
    scalar multiplication & $\otimes$ \\
    a basis $B$ & an object $B$\\
    a vector $v$ in the basis $B$ & a morphism $v : X\rightarrow B$\\
    the vector space with basis $B$ & the slice category $\mathcal{C}_{/B}$\\
    vector addition $v + w$ & coproduct $v\oplus w : X\oplus Y\rightarrow B$\\
    scalar multiplication $av$ & $A \otimes (v : X\rightarrow B) := (A\otimes X\xrightarrow{\operatorname{id}\otimes v} A\otimes B\xrightarrow{s\otimes \operatorname{id}_{B}}B)$\\
    a matrix $M$ & a span $\roofx{B}{C}{M}{v}{w}$\\
    the linear map with matrix $M$ & the linear functor $w_{!}v^{*} : \mathcal{C}_{/B}\rightarrow\mathcal{C}_{/C}$\\
    matrix multiplication & span composition\\
    \hline
\end{tabular}
\end{center}

\begin{ex}
Fix the base field $\F_q$ and an algebraic closure $\overline{\F}$ of $\F_q$. Let $G = \Gal(\overline{\F}/\F_q)$ be the absolute Galois group of $\F_q$; it is isomorphic to the profinite completion of the integers $\widehat{\Z}$, and is topologically generated by the $q$-power Frobenius $\Fr_q$. We define the category $G$-\cat{Rep} to have objects which are complex vector spaces which have a continuous action of $G$, with $G$-equivariant morphisms. For our purposes, its terminal object is not interesting, as it is the zero vector space. So we use instead the slice category $G$-mod$_{/\C[G]}$. 

The coproduct in $G$-\cat{Rep} is the direct sum $\oplus$, and it can be made into a symmetric monoidal category by way of the tensor product $\oplus$.
\end{ex}

\subsection{Using $G$-Representations}
\label{app:Grep}

The main advantage to using a category such as $G$-\cat{Rep} over $\cat{Set}$ is that objects in $G$-\cat{Rep} have richer decategorifications than those in $\cat{Set}$. Take, for instance, an ordinary $2\times 2$ matrix $A = \begin{pmatrix} a & b \\ c & d\end{pmatrix}$. If $a,b,c,d$ are all nonnegative integers, then $A$ is easily lifted to an objective matrix 
$$
A = \left (\roofx{B}{B}{X}{s}{t}\right )
$$
where $B = \{1,2\}$ and $X = X_{a}\amalg X_{b}\amalg X_{c}\amalg X_{d}$, with $|X_{i}| = i$ and $s$ (resp.~$t$) sending elements in $X_{a},\ldots,X_{d}$ to the number of the column (resp.~row) their label appears in in $A$. However, if any of $a,b,c,d$ are negative, rational, real, complex or beyond, it is only possible to lift $A$ to an objective matrix over a more general category. 

For an explicit example, take 
$$
A = \begin{pmatrix} 1 & -3 \\ -2 & -1\end{pmatrix}. 
$$
In the category $C_{2}$-\cat{Rep} (say, over $\C$), there are two irreducible representations up to isomorphism: $V_{1}$ and $V_{-1}$, where $V_{d}$ is $1$-dimensional with the nontrivial element in $C_{2}$ acting as multiplication by $d$. Set $V = V_{1}\oplus V_{-1}$ and consider the objective matrix 
$$
A = \left (\roofx{V\oplus V}{V\oplus V}{W}{s}{t}\right )
$$
where $W = V_{1}\oplus V_{-1}^{\oplus 3}\oplus V_{-1}^{\oplus 2}\oplus V_{-1}$ and $s$ (resp.~$t$) sends the first and third factors $V_{1}$ and $V_{-1}^{\oplus 2}$ (resp.~the first two factors $V_{1}$ and $V_{-1}^{\oplus 3}$) onto their corresponding factors in the first component of $V\oplus V$, and the remaining factors to their corresponding factors in the second component. Taking the trace of the action of a generator of $C_{2}$ on components recovers the original matrix $A$. If $A$ had entries which were roots of unity, we could replace $C_{2}$ with the appropriate cyclic group $C_{n}$ and recover these entries as well. For more general entries, more general structures are required (i.e.~Galois representations over an appropriate extension of $\Q$ containing the coefficients of the matrix, but possibly more generality is also required). 

In this more general framework, it is possible to categorify formula (\ref{eq:Lfactor}) without splitting up the $L$-function, in contrast to the formula in Theorem~\ref{thm:mainthm}. In short, negative coefficients must be passed off as positive coefficients on the other side of an objective formula over $\cat{Set}$, but they can be represented directly over $C_{2}$-\cat{Rep} using $V_{-1}$. We will describe a vast generalization of this procedure in a future article. For now, we give a brief description of an $L$-functor approach to formula (\ref{eq:Lfactor}). 

\subsection{Objective $L$-Polynomials}
\label{app:Lpoly}

In this section, we will show how the $L$-polynomial $L(E,t) = 1 - a_{q}t + qt^{2}$ of an elliptic curve $E/\F_{q}$ can be interpreted cohomologically in the objective setting. It is well-known that $L(E,t)$ is a characteristic polynomial, namely $\det(1 - tF)$ where $F = \Frob_{q}$ is the Frobenius operator acting on $H^{1}(E,\Q_{\ell})$. There is a general notion of characteristic polynomial in objective linear algebra, but we will only need the following version. 

As in Section~\ref{app:Grep}, let $G = \Gal(\overline{\F}_{q}/\F_{q})$ and let $V$ be any continuous $G$-representation. Define the {\it $L$-functor} of $V$ to be the span 
$$
L(V) = \left (\roof{\bigoplus_{n\geq 0} \Q_{\ell}}{\Q_{\ell}}{\bigoplus_{n\geq 0} \Sym^{n}V}\right ). 
$$
This defines an element in the incidence algebra of $\bigoplus_{n = 0}^{\infty} \Q_{\ell}$, viewed as a simplicial $G$-representation with trivial structure in each level. When $V$ is finite dimensional, taking traces of Frobenius yields the following sequence in the numerical incidence algebra $I_{\#}(\N_{0};\Q_{\ell})$: 
$$
|L(V)|(n) := \Tr(F\mid \Sym^{n}V). 
$$
These assemble into the \emph{reciprocal} characteristic polynomial of $V$: 
$$
\sum_{n = 0}^{\infty} |L(V)|(n)t^{n} = \sum_{n = 0}^{\infty} \Tr(F\mid \Sym^{n}V)t^{n} = \frac{1}{\det(1 - tF\mid V)}. 
$$

In the case of an elliptic curve $E/\F_{q}$, let $V = H^{1}(E,\Q_{\ell})$ so that $\det(1 - tF\mid V) = L(E,t)$. Remark~\ref{rem:card} also implies that 
$$
L(E,t) = \sum_{n = 0}^{\infty} (\chi^{+}(n) - \chi^{-}(n))t^{n}
$$
so as a consequence, the arithmetic functions $L(E) := L(E)^{+} - L(E)^{-}$ and $|L(V)|$ are inverses in the numerical incidence algebra $I_{\#}(\N_{0};\Q_{\ell})$. In fact, these are inverses in an objective incidence algebra. To show this, we have to realize $L(E)$ in the category of $G$-representations. Using the notation from Section~\ref{sec:mainthmglob}, define 
$$
L(E) = \left (\roof{\bigoplus_{n\geq 0} \Q_{\ell}}{\Q_{\ell}}{\Q_{\ell}S_{1}^{+}\oplus\Q_{\ell}S_{1}^{-}}\right )
$$
where $F = \Frob_{q}$ acts on $\Q_{\ell}S_{1}^{+}$ by the identity and on $\Q_{\ell}S_{1}^{-}$ by $-1$. Then in the objective incidence algebra $I\left (\bigoplus_{n\geq 0} \Q_{\ell}\right )$, we get a convolution $L(V)*L(E)$, which has the same trace generating function\footnote{In general, we cannot expect linear functors to have objective inverses \cite{gkt2}.} as the unit functor $\delta$. This shows that we can view $L(V)$ is an objective lift of the reciprocal $L$-polynomial of $E$.

\section{Incidence Algebras and Products}
\label{app:tensor}

In the literature, it is often implicitly assumed that for two monoids (or posets) $X$ and $Y$, the numerical incidence algebra of their product is the tensor product of their incidence algebras: $I(X\times Y) \cong I(X)\otimes I(Y)$. This is only true for \emph{finite} monoids and posets: 

\begin{prop}
Let $X$ and $Y$ be finite monoids or posets. Then $I(X\times Y)\cong I(X)\otimes I(Y)$. 
\end{prop}

\begin{proof}
The product $X\times Y$ is again a finite monoid (resp.~poset) and we have a map 
\begin{equation}\label{eq:tensordual}
I(X)\otimes_{k}I(Y) = k[X_{1}]^{*}\otimes_{k}k[Y_{1}]^{*} \longrightarrow (k[X_{1}]\otimes_{k}k[Y_{1}])^{*} = k[X_{1}\times Y_{1}]^{*} = k[(X\times Y)_{1}]^{*} = I(X\times Y)
\end{equation}
which sends $\sum f_{i}(x)\otimes g_{i}(y)$ to $\sum f_{i}(x)g_{i}(y)$. This is clearly injective and both source and target have the same (finite) dimension, so the map is an isomorphism. 
\end{proof}

When $X$ and $Y$ are infinite, the map (\ref{eq:tensordual}) need not be an isomorphism. 

\begin{ex}
Let $X = Y = \N_{0}$ and denote their incidence algebras by $I(X) = k[[s]]$ and $I(Y) = k[[t]]$, respectively. Then $I(X\times Y) = I(\N_{0}^{2}) \cong k[[s,t]]$ by sending $f : \N_{0}^{2}\rightarrow k$ to its generating function 
$$
F(s,t) = \sum_{i,j\geq 0} f(i,j)s^{i}t^{j}. 
$$
However, it is well-known that $k[[s]]\otimes_{k}k[[t]]$ is not isomorphic to $k[[s,t]]$. Instead, $k[[s,t]]$ is isomorphic to the {\it completed tensor product} of $k[[s]]$ and $k[[t]]$, whose definition we recall now. 
\end{ex}

\begin{defn}
Let $A$ and $B$ be two linearly topologized $k$-vector spaces, say with topologies defined by collections of ideals $\{I_{\alpha}\}$ and $\{J_{\beta}\}$, respectively. The {\bf completed tensor product} $A\widehat{\otimes}_{k}B$ is the inverse limit 
$$
A\widehat{\otimes}_{k}B := \invlim (A/I_{\alpha}\otimes_{k}B/J_{\beta})
$$
over all $I_{\alpha}\subset A$ and $J_{\beta}\subset B$ in the respective fundamental systems of neighborhoods. This is a topological $k$-vector space with the profinite topology. 
\end{defn}

\begin{lem}
There is an isomorphism of topological $k$-algebras 
$$
k[[s]]\widehat{\otimes}_{k}k[[t]] \cong k[[s,t]],
$$
where $k[[s]]$ (resp.~$k[[t]],k[[s,t]]$) has the $(s)$-adic (resp.~$(t)$-adic, $(s,t)$-adic) topology. 
\end{lem}

\begin{proof}
For each $m,n\geq 1$, we have an isomorphism $k[[s]]/(s^{m})\otimes_{k}k[[t]]/(t^{n}) \cong k[[s,t]]/(s^{m},t^{n})$ which preserves multiplication and is compatible with the projections $k[[s,t]]/(s^{m'},t^{n'})\rightarrow k[[s,t]]/(s^{m},t^{n})$ for any $m'\geq m,n'\geq n$. This, together with the isomorphism 
$$
k[[s,t]] \cong \invlim k[[s,t]]/(s^{m},t^{n}),
$$
implies $k[[s]]\otimes_{k}k[[t]] \cong k[[s,t]]$ as topological $k$-algebras. 
\end{proof}

For a locally finite decomposition set $X$, its incidence algebra $I(X)$ is naturally a topological $k$-algebra under the standard topology defined by pointwise convergence, i.e.~$(f_{n})\subset I(X)$ converges if and only if $(f_{n}(x))\subset k$ converges for all $x\in X_{1}$. This topology has the following equivalent description. For each finite set $F\subseteq X_{1}$, let $\U_{F}\subseteq I(X)$ be the subspace of functions vanishing on $F$: 
$$
\U_{F} = \{f\in I(X) \mid f(x) = 0 \text{ for all } x\in F\}. 
$$

\begin{lem}
For any locally finite decomposition set $X$, $\{\U_{F}\}$ is a fundamental system of neighborhoods of $0$ in $I(X)$. In particular, $I(X)$ is a (right) linearly topologized $k$-algebra. 
\end{lem}

\begin{proof}
For a finite subset $F\subseteq X_{1}$, define 
$$
F' = \{d_{2}\sigma\mid \sigma\in X_{2} \text{ with } d_{1}\sigma\in F\}. 
$$
Then $F'$ is finite since $X$ is locally finite (in particular, because the maps $d_{i} : X_{2}\rightarrow X_{1}$ are finite) and $F\subseteq F'$ since for any $x\in F$, the degenerate $2$-simplex $\sigma = s_{1}x$ satisfies $d_{2}\sigma = x$ and $d_{1}\sigma = d_{1}s_{1}x = x\in F$. We claim $\U_{F'}$ is a right ideal of $I(X)$. To show this, we first prove $F'' = F'$, for which we only need to prove $F''\subseteq F'$. For $x\in F''$, take $\sigma\in X_{2}$ with $d_{2}\sigma = x$ and $y := d_{1}\sigma\in F'$. Then there is some $\tau\in X_{2}$ with $d_{2}\tau = y$ and $z := d_{1}\tau\in F$, as pictured below. 
\begin{center}
\begin{tikzpicture}[scale=2]
    \node at (0,0) (0) {$\bullet$};
    \node at (.5,.7) (1) {$\bullet$};
    \node at (1,0) (2) {$\bullet$};
    \node at (.5,-.7) (3) {$\bullet$};
    \draw[->,thick,dashed,red] (1) -- (3) node[right,pos=.7] {$w$};
    \draw[line width=20,white] (.2,.15) -- (.8,.15);
    \draw[->] (0) -- (2) node[above,pos=.5] {$y$};
    \draw[->] (0) -- (1) node[left,pos=.5] {$x$};
    \draw[->] (1) -- (2) node[right,pos=.5] {$y'$};
    \draw[->] (0) -- (3) node[left,pos=.5] {$z$};
    \draw[->] (2) -- (3) node[right,pos=.5] {$z'$};
\end{tikzpicture}
\end{center}
Since every decomposition set is $2$-Segal \cite[Rmk.~3.2]{gkt1}, there is a unique $3$-simplex $A$ obtained by filling in the red $1$-simplex in the figure above, so that $d_{3}A = \sigma$ and $d_{1}A = \tau$. The face $\nu := d_{2}A$, shown above with faces $x,w$ and $z$, witnesses the membership of $x$ in $F'$, since $d_{2}\nu = x$ and $d_{1}\nu = z\in F$. This implies $\U_{F'}$ is a (right) ideal, since for any $f\in\U_{F'}$, $g\in I(X)$ and $x\in F'$, we have 
\begin{equation}\label{eq:leftideal}
(f*g)(x) = \sum_{d_{1}\sigma = x} f(d_{2}\sigma)g(d_{0}\sigma) = \sum_{d_{1}\sigma = x} 0\cdot g(d_{0}\sigma) = 0. 
\end{equation}
As a result, the $\U_{F'}$ form a fundamental system of neighborhoods of $0$ which are (right) ideals of $I(X)$. 
\end{proof}

\begin{ex}
The argument above using the $2$-Segal condition is best illustrated in an example. Let $\mathcal{C}$ be a locally finite category, or equivalently, a locally finite $1$-Segal set, which is also $2$-Segal by \cite[Prop.~1.3]{boors}. Then a morphism $x$ lies in $F''$ if there is some composition $y = y'\circ x$ with $y\in F'$. In turn, this means there is some decomposition $z = z'\circ y$ with $z\in F$. The composition $z = z'\circ (y'\circ x) = (z'\circ y')\circ x$ lies in $F$, so $x\in F'$. Hence $F''\subseteq F'$. In addition, the locally finite condition ensures that there are only a finite number of compositions of every morphism in $\mathcal{C}$, so the sums in equation (\ref{eq:leftideal}) are finite. 
\end{ex}

\begin{rem}
One can construct a fundamental system of neighborhoods of $0$ consisting of \emph{left} ideals of $I(X)$ by replacing $F'$ above with its ``right companion''
$$
F' = \{d_{0}\sigma \mid \sigma\in X_{2} \text{ with } d_{1}\sigma\in F\}. 
$$
(Here, ``right'' refers to the right face $d_{0}\sigma$ of a $2$-simplex $\sigma$, which produces a left ideal $\U_{F'}$.) If, as in most situations in this article, $X$ is a commutative monoid, then $I(X)$ is commutative ring and in particular is linearly topologized by two-sided ideals. 
\end{rem}


\begin{lem}
\label{lem:incalgcomplete}
$I(X)$ is complete with respect to the topology induced by $\{\U_{F}\}$. 
\end{lem}

\begin{proof}
We need to show that 
$$
I(X) \cong \invlim I(X)/\U_{F},
$$
where the limit runs over all finite subsets $F\subseteq X_{1}$, directed by inclusion. 
Each $I(X)/\U_{F}$ naturally identifies with the set of functions on $F$: 
\begin{align*}
    I(X)/\U_{F} &\xrightarrow{\;\sim\;} \Hom(F,k)\\
    f + \U_{F} &\longmapsto f|_{F}. 
\end{align*}
This is clearly well-defined and compatible with the projections $I(X)/\U_{F'}\rightarrow I(X)/\U_{F}$ and $\Hom(F',k)\rightarrow\Hom(F,k)$ for $F'\supseteq F$. Then since $\bigcup F = X_{1}$, we have 
$$
I(X) = \Hom(X_{1},k) = \Hom\left (\bigcup F,k\right ) \cong \invlim \Hom(F,k) \cong \invlim I(X)/\U_{F},
$$
as required. 
\end{proof}

\begin{lem}
The standard topology on $I(X)$ is equivalent to the topology induced by $\{\U_{F}\}$. 
\end{lem}

\begin{proof}
In the topology of pointwise convergence, a sequence $(f_{n})\subset I(X)$ convergences if and only if $(f_{n}(x))$ converges for all $x\in X_{1}$, but since $k$ has the discrete topology, this is the same as $(f_{n}(x))$ being eventually constant for all $x\in X_{1}$. By Lemma~\ref{lem:incalgcomplete}, $I(X) \cong \invlim I(X)/\U_{F}$ where $F$ ranges over all finite subsets of $X_{1}$, and vanishing on some finite set $F$ is equivalent to being eventually constant, so we see that the standard topology on $I(X)$ is equivalent to the profinite topology. 
\end{proof}

\begin{prop}
Let $X$ and $Y$ be locally finite decomposition sets. Then $I(X\times Y) \cong I(X)\widehat{\otimes}_{k}I(Y)$ as topological $k$-algebras. 
\end{prop}

\begin{proof}
By Lemma~\ref{lem:incalgcomplete}, it suffices to construct compatible isomorphisms 
$$
I(X)/\U_{F}\otimes_{k}I(Y)/\U_{G} \cong I(X\times Y)/\U_{F\times G}
$$
for all finite subsets $F\subseteq X_{1}$ and $G\subseteq Y_{1}$. Such maps are given by 
\begin{align*}
    I(X)/\U_{F}\otimes_{k} I(Y)/\U_{G} \cong \Hom(F,k)\otimes_{k}\Hom(G,k) &\xrightarrow{\;\psi_{F,G}\;} \Hom(F\times G,k) \cong I(X\times Y)/\U_{F\times G}\\
    f\otimes g &\longmapsto f\bullet g
\end{align*}
where $(f\bullet g)(x,y) = f(x)g(y)$. These are visibly linear and compatible with the projection maps on both sides. Since $F$ and $G$ are finite, $\Hom(F,k),\Hom(G,k)$ and $\Hom(F\times G,k)$ are all spanned by indicator functions, which implies each $\psi_{F,G}$ is surjective. Finally, dimension counting shows each $\psi_{F,G}$ is an isomorphism, so they assemble into an isomorphism $\psi : I(X)\widehat{\otimes}_{k}I(Y) \xrightarrow{\sim} I(X\times Y)$ as desired. 
\end{proof}


\end{document}